\documentclass[11pt,thmsa]{article}

\usepackage{amssymb,latexsym}
\usepackage{amsmath,amscd,amsthm, amssymb, latexsym}
\usepackage[all]{xy}

 \usepackage[latin1]{inputenc}
 \usepackage{mathrsfs}
 \usepackage{dsfont}
 \usepackage{amsmath}
 \usepackage{amsthm}
 \usepackage{amsfonts}
 \usepackage{amssymb}
 \usepackage[dvips]{graphicx}
 \usepackage{wrapfig}
 \usepackage{layout}
 \usepackage{verbatim}
 \usepackage{alltt}
 \usepackage{xypic}
 \input xy
 \xyoption{all}

 \DeclareMathAlphabet{\mathpzc}{OT1}{pzc}{m}{it}

 \newtheorem{theorem}{Theorem}[section]
 \newtheorem{lemma}[theorem]{Lemma}
 \newtheorem{proposition}[theorem]{Proposition}
 
 \newtheorem{corollary}[theorem]{Corollary}
 \newtheorem{definition}[theorem]{Definition}
  \theoremstyle{definition}
 \newtheorem{example}[theorem]{Example}
 \newtheorem{remark}[theorem]{Remark}

\newtheorem*{theorem1}{Theorem \ref{vanishing}}
\newtheorem*{theorem2}{Theorem \ref{main theorem}}
\newtheorem*{theorem3}{Corollary \ref{vanishingBS}}
\newtheorem*{theorem4}{Theorem \ref{theorem spectral sequence}}

\newtheorem*{acknowledgements}{Acknowledgements}

\newcommand{\rmap}{\longrightarrow}
\newcommand{\lmap}{\longleftarrow}

\newcommand{\uHom}{\mathrm{Hom}}

\newcommand{\uRep}{\mathrm{Rep}}
\newcommand{\uRRep}{\mathcal{R}\mathrm{ep}}

\newcommand{\uAd}{\mathrm{Ad}}

\newcommand{\uC}{C}

\newcommand{\uHdiff}{\mathrm{H}_{\mathrm{diff}}}

\newcommand{\derham}{\mathrm{d_{DR}}}
\newcommand{\udh}{\mathrm{d_h}}
\newcommand{\uEnd}{\mathrm{End}}
\newcommand{\uS}{\mathcal{S}}
\newcommand{\uId}{\mathrm{Id}}

\begin{document}
\vspace{15cm}
 \title{Representations up to homotopy and Bott's spectral sequence for Lie groupoids  }
\author{ Camilo Arias Abad\footnote{Institut f\"ur Mathematik, Universit\"at Z\"urich,
camilo.arias.abad@math.uzh.ch. Partially supported by NWO Grant ``Symmetries and Deformations in Geometry'' and
by SNF Grant 20-113439. } and Marius Crainic\footnote{Mathematics Institute, Utrecht University, M.Crainic@math.uu.nl. Partially supported by NWO Grant ``Symmetries and Deformations in Geometry''.}\\
} \maketitle
 \maketitle
 \abstract{}

This belongs to a series of papers devoted to the study
of the cohomology of classifying spaces of Lie groupoids. 
Our aim here is to introduce and study the notion of representation up to homotopy of Lie groupoids, the resulting derived category,
and to show that the adjoint representation is well defined as a representation up to homotopy.
As an application, we extend Bott's spectral sequence converging to the
cohomology of classifying spaces from Lie groups to Lie groupoids. 
Our work is closely related to and inspired by Behrend
\cite{BEH1}, Bott \cite{Bott}, and Getzler \cite{Get}.

 \tableofcontents

\section{Introduction}

This work belongs to a series of papers whose original motivation was the study of the cohomology of classifying spaces
of Lie groupoids. Our aim is to introduce and study the notion of representation up to homotopy of Lie groupoids; as an application,
we show how the Bott spectral sequence converging to the cohomology of the classifying space can be extended 
to general Lie groupoids.

Lie groupoids are intimately related to various branches of geometry and topology, such as 
the transversal geometry of foliations, Lie theory  and
Poisson geometry, to mention a few. Also, via their convolution algebras, they are 
the main source of examples of noncommutative spaces. Lie groupoids can  be 
viewed as global counterparts of Lie algebroids, in the same way that Lie groups are
related to Lie algebras. They very often arise after an integration process
(simplest example: the flow of a vector field). \\

{\bf Representations up to homotopy:} 
One of the two main themes of this paper is the concept of representation up to homotopy.
This notion is new even in the case of Lie groups; while the notion of representation 
has a clear extension to the framework of Lie groupoids, there are in general very
few such representations. For instance, in the context of ordinary representations there is no good notion of an
adjoint representation- an object which, as we will see, is important for understanding the cohomology
of classifying spaces. What we propose here is the notion of representation up
to homotopy and the resulting categories $\uRRep^{\infty}(G)$ associated to groupoids $G$.
We also point out that, algebraically, this is just a shadow of a DG category, and a derived
category $\mathcal{D}(G)$ appears as a more appropriate invariant of $G$. 

The idea behind representations up to homotopy is to represent the groupoid $G$ not in a vector bundle, but in a complex of vector bundles
\[ E^{\bullet}: \ \ \ldots \stackrel{\partial}{\rmap} E^{0}\stackrel{\partial}{\rmap} E^1 \stackrel{\partial}{\rmap} E^1 \stackrel{\partial}{\rmap}\ldots .\]
There are two reasons for using complexes instead of vector bundles: they encode objects which are singular
(when the cohomology bundle associated to $E^{\bullet}$ is not smooth) and they 
allow more flexibility in the associativity requirements. Both features are important in order to define a good
adjoint representation (see Remark \ref{reasons-complexes} for more details). Indeed we conclude that the adjoint representation
of a Lie groupoid is a representation up to homotopy.

In a representation up to homotopy, an arrow $ g:x\rightarrow y$ acts as a map of complexes
\begin{equation*}
\lambda_g:E^{\bullet}_x \rightarrow E^{\bullet}_y,
\end{equation*}
but we allow this action not to respect the associativity. That is,
in general $\lambda_{g_1 g_2}$ and $\lambda_{g_1} \lambda_{g_2}$ are
not the same map of complexes. However, they are homotopic
maps, and there is a controlled and coherent way of choosing the
homotopies (``curvatures''). The definition of representation up to homotopy is the cochain
version of the classical notion of strongly homotopy multiplicative map,
as defined in Stasheff's book \cite{stasheff}.

Finally, we would like to mention that representations up to homotopy of Lie groupoids are
the global version of the infinitesimal case discussed in \cite{AC1} which, in turn, is based on
Quillen's superconnections \cite{Q}. However, things become more intricate at the global level-
and probably the main reason is the fact that, while the infinitesimal discussion is governed by
De-Rham algebras (associated to algebroids), which are commutative DGA's, the global one
is governed by a DGA which is commutative only up to homotopy (the algebra
of smooth groupoid cochains). This is reflected in the fact that the tensor
product of representations up to homotopy of groupoids is much more subtle and hides rather 
surprising combinatorics of topological origin. Although we point out here what happens with 
the tensor products, the more precise and complete treatment is deferred to \cite{ACD}.\\

{\bf The Bott spectral sequence:} The other main theme of this paper is a general Bott spectral sequence for Lie groupoids,
converging to the cohomology of classifying spaces. Let us first recall the case of Lie groups. 
The cohomology of the classifying space $BG$ of a Lie group $G$ can
be computed as the total cohomology of a double complex $\Omega(G_{\bullet})$, which
consists of differential forms on products of the group $G$. This space
has a vertical differential which is just the De-Rham operator, and a horizontal
differential which is given by the simplicial structure on the nerve of $G$.
In \cite{Bott}, Bott computed the horizontal cohomologies of this
double complex by proving the formula:
\begin{equation}\label{formula of Bott}
H_{\mathrm{d_h}}^p(\Omega^q(G_{\bullet}))\cong
H^{p-q}_{\mathrm{diff}}(G;S^q(\mathfrak{g}^*)),
\end{equation}
where the right hand side is the differentiable cohomology with
coefficients in the symmetric powers of the coadjoint
representation. Filtering the double complex by the differential
form degree, he obtained a spectral sequence converging to the cohomology
of $BG$:
\begin{equation}\label{spectral sequence Bott}
 E_1^{pq}=H^{p-q}_{\mathrm{diff}}(G;S^q(\mathfrak{g^*}))\Rightarrow
 H^{p+q}(BG).
 \end{equation}
In particular, due to the degree shift in the formula,
the first page of the spectral sequence vanishes
above the diagonal. On the other hand, it is well known that
for any finite dimensional representation $E$ of
a compact Lie group $G$:
\begin{equation*}
H^k_{\mathrm{diff}}(G;E)=0, \qquad \text{for $k>0$.}
\end{equation*}
In conclusion, in the
compact case, the $E_1$ terms of the spectral sequence vanish
outside the diagonal and Bott concludes:
\begin{equation*}
 H^{2p}(BG) \cong S^p(\mathfrak{g^*})^G.
\end{equation*}

In \cite{Get}, Getzler constructed a model for equivariant
cohomology in the non compact case, generalizing in particular Bott's spectral
sequence; this can be seen as a generalization of Bott's discussion from Lie
groups to groupoids associated to Lie group actions on manifolds. 
Behrend \cite{BEH1} generalized Getzler's model to the case of flat groupoids.
In this paper we generalize Bott's formula (\ref{formula of Bott}) and spectral
sequence (\ref{spectral sequence Bott}) to the case of an arbitrary Lie
groupoid $G$.\\

{\bf Main results:} The results of this paper can be divided into two types.
On one hand, we have the general constructions: defining  the notion
of representations up to homotopy (and the resulting derived category) 
as well as the construction of the adjoint representation.
On the other hand, we have several theorems, some of which we mention here.
First of all, there is a vanishing result for the differentiable cohomology associated to representations 
up to homotopy of proper groupoids (the analogue of compact groups). Such results are useful 
for instance in deformation problems (see also the remarks at the beginning of
subsection \ref{The vanishing theorem}). 
\begin{theorem1}
Let $E=\bigoplus_{l=a}^b E^l$ be a unital representation up to
homotopy of a proper Lie groupoid $G$. Then:
\begin{equation*}
\uHdiff^q(G;E) =0, \text{         if } q \notin [a,b].
\end{equation*}
\end{theorem1}

Next, we generalize Bott's formula (\ref{formula of Bott})
to the case of an arbitrary Lie
groupoid $G$.

\begin{theorem2}
Let $G$ be a Lie groupoid. Then, for the cohomology of the horizontal complex
\begin{equation*}
\xymatrix{
\Omega^q(G_0)\ar[r]^-{d_h}&\Omega^q(G_1)\ar[r]^-{d_h}&\cdots
\ar[r]^-{d_h}& \Omega^q(G_p)\ar[r]^-{d_h}&\cdots, }
\end{equation*}
one has
\begin{equation}\label{main formula}
H_{d_h}^p(\Omega^q(G_{\bullet}))\cong
H_{\mathrm{diff}}^{p-q}(G;\uS^q\uAd^*).
\end{equation}
\end{theorem2}

Note that, in order to give meaning to the right hand side of equation (\ref{main formula}),
we are using the adjoint representation as a representation up to homotopy and its symmetric powers.
One clarification is necessary here: the construction of the spectral sequence only depends on some of the properties of 
the symmetric powers of representations up to homotopy. So, while the general construction of the symmetric power operation
(which comes together with the construction of tensor products already mentioned) is deferred to
\cite{ACD}, we decided to work axiomatically  by stating the properties that we require from the these operations.
The axioms we use are just {\it some} of the  properties that are satisfied by the construction given in loc. cit; in particular, 
we do not need here the construction of $\uS^q$ when applied to morphisms between representations up to homotopy. 
We decided to work axiomatically 
in order to simplify the presentation and to make it independent of the
quite different problem of understanding the symmetric powers; also, we would like 
not to rule out other constructions of symmetric powers  which may appear in specific situations.

Combining Theorem \ref{main theorem} and Theorem \ref{vanishing} one obtains the following
result which, as pointed to us by  Ping Xu,
implies the existence of connections on $S^1$-gerbes over separated stacks.

\begin{theorem3}
Let $G$ be a proper Lie groupoid. Then
\begin{equation*}
H_{d_h}^p(\Omega^q(G_{\bullet}))=0, \text{ if } p>q.
\end{equation*}
\end{theorem3}

Finally, considering the filtration in the Bott-Shulman-Stasheff complex for the cohomology of $BG$ one also
concludes:
\begin{theorem4}
Let $G$ be a Lie groupoid. There is a spectral sequence converging
to the cohomology of $BG$:
\begin{equation*}
E_1^{pq}=H_{\mathrm{diff}}^{p-q}(G; \uS^q(\uAd^*))\Rightarrow
H^{p+q}(BG).
\end{equation*}
Moreover, if the groupoid is proper, the $E_1$  terms of the spectral sequence
vanish bellow the diagonal.
\end{theorem4}

\begin{acknowledgements}
We would like to thank Ieke Moerdijk and James Stasheff for their comments related to this work.
Also, we would like to thank B. Dherin for the many discussions we had, which greatly
influenced the final outcome.
\end{acknowledgements}

\section{Preliminaries}

\subsection{Groupoids}
\label{oids}

In this subsection we review some known facts about Lie groupoids and actions. As a general reference, we use \cite{McK}.

{\bf Lie groupoids and Lie algebroids:}  A groupoid is a category in which all arrows are isomorphisms. A Lie
groupoid is a groupoid in which the space of objects $G_0$ and the
space of arrows $G_1$ are smooth manifolds and all the structure
maps are smooth. More explicitly, a Lie groupoid is given by a
manifold of objects $G_0$ and a manifold of arrows $G_1$ together
with smooth maps $s,t:G_1 \rightarrow G_0$ the source and target
map, a composition map $m:G_1\times_{G_0} G_1\rightarrow G_1$, an
inversion map $\iota:G\rightarrow G$  and an identity map
$\epsilon:G_0 \rightarrow G_1$ that sends an object to the
corresponding identity. These structure maps should satisfy the usual
identities for a category. The source and target maps are required
to be surjective submersions and therefore the domain of the composition map 
is a manifold. We will usually denote the space of objects of a
Lie groupoid by $M$ and say that $G$ is a groupoid over $M$.

Recall also that a Lie algebroid over a manifold $M$ is a vector bundle $\pi:A \rightarrow
M$ together with a bundle map $\rho:A \rightarrow TM$, called the
anchor map and a Lie bracket on the space $\Gamma(A)$ of sections of
$A$ satisfying Leibniz identity:
\begin{equation*}
[\alpha,f\beta]=f[\alpha,\beta]+\rho(\alpha)(f)\beta,
\end{equation*}
for every $\alpha,\beta \in \Gamma(A)$ and $f \in C^{\infty}(M)$. It
follows that $\rho$ induces a Lie algebra map at the level of
sections. 
Lie algebroids are the infinitesimal counterparts of Lie groupoids.
Given a Lie groupoid $G$, its Lie algebroid $A= A(G)$ is defined as follows.
As a vector bundle, it is  the restriction of the kernel of
the differential of the source map to $M$. Hence, its fiber at 
$x\in M$ is the tangent space at the identity arrow $1_x$ of the
source fiber $s^{-1}(x)$.  The
anchor map is the differential of the target map. To describe the bracket,
we need to discuss invariant vector fields. A right invariant vector field on a Lie groupoid $G$ is a vector field
$\alpha$ which is tangent to the fibers of $s$ and such that, if
$g,h$ are two composable arrows and we denote by $R^{\textrm{h}}$ the right
multiplication by $h$, then
\begin{equation*}
\alpha (gh)=D_g(R^{\textrm{h}})(\alpha (g)).
\end{equation*}
It is not difficult to see that the space of right invariant vector fields is closed under the Lie bracket of vector fields, on one hand, 
and is isomorphic to $\Gamma(A)$, on the other hand. All together, we get the desired Lie bracket on $\Gamma(A)$. \\

\begin{example}\rm \ 
A Lie group $G$ can be seen as a Lie groupoid in which the
space of objects is a point. Associated to any manifold $M$ there is the pair groupoid $M\times M$ over $M$
for which there is exactly one arrow between each pair of points. Associated to any foliation is a Lie groupoid,
the holonomy groupoid of the foliation, which arises as the ``smallest desingularization of the leaf space''.
If a Lie group $G$ acts on a manifold $M$ there is an
associated action groupoid over $M$ denoted $G \ltimes M$ whose
space of arrows is $G\times M$ (see also below).

Infinitesimally, examples of Lie algebroids are Lie algebras, tangent
bundles, Poisson manifolds, foliations and Lie algebra actions.

Unlike
the case of Lie algebras, Lie's third theorem does not hold in this generality: not every Lie algebroid can be integrated
to a Lie groupoid. The precise conditions for the integrability are
described in \cite{CF1}. However, Lie's first and second theorem do hold. Due to the first one- which says that
if a Lie algebroid is integrable then it admits a
canonical source simply connected integration- one may often assume that the Lie groupoids under discussion
satisfy this simply-connectedness condition. Examples of groupoids which arise via Lie algebroids
are the symplectic groupoids of Poisson geometry and the monodromy groupoids 
of foliation theory.
\end{example}

{\bf Actions:} A left action of a groupoid $G$ on a space $P \stackrel{\nu}{\rightarrow} M$ over
$M$ is a map $G_1\times_{M} P \rightarrow P$ defined on the space $G_1\times_{M} P$ of
pairs $(g, p)$ with $s(g)= \nu(p)$, which satisfies $\nu(gp)= t(g)$ and the usual
conditions for actions. Associated to the action of $G$ on
$P\rightarrow M$ there is the action groupoid, denoted $G\ltimes P$. The base space is $P$, 
the space of arrows is $G_1\times_{M}P$ we have just mentioned, the source map is the second projection, while the target map
is the action. The multiplication in this groupoid is $(g, p)(h, q)= (gh, q)$.

This discussion has an infinitesimal version. Let $A$ be an algebroid over $M$. An infinitesimal action of $A$ on $P \stackrel{\nu}{\rightarrow} M$ is
given by a vector bundle map $\tilde{\rho}_P: \nu^*A\rmap TP$ with the property that $(d\nu)\circ \tilde{\rho}= \rho$ (the anchor of $A$)
and the induced map at the level of sections, $\tilde{\rho}_P: \Gamma(A)\rmap \mathcal{X}(P)$ is a Lie algebra map. Associated
to such an infinitesimal action there is the action algebroid $A\ltimes P$ over $P$. As a vector bundle, it is just $\nu^*A$, the anchor is
the infinitesimal action, while the bracket is determined by the Leibniz identity and $[\nu^*\alpha, \nu^*\beta]= \nu^*[\alpha, \beta]$. 
If a Lie groupoid $G$ acts on $P$, then there is an induced action of its Lie algebroid $A$ on $P$, and the Lie algebroid of $G\ltimes P$ is $A\ltimes P$.

\begin{example}\label{action-nerve}\rm \ 
For a Lie groupoid $G$, we denote by $G_k$ the space of strings of
$k$ composable arrows of $G$. When we
write a string of $k$ composable arrows $(g_1,\dots,g_k)$ we mean
that $t(g_i)=s(g_{i-1})$. Since the source and target maps are submersions, all the $G_k$ 
are manifolds. Each of the $G_k$'s carries a natural left action. First of all, we view $G_k$ over $M$ via the map
\[ t: G_k\rightarrow M, \ \ (g_1, \ldots , g_k)\mapsto t(g_1).\]
The left action of $G$ on $G_k\stackrel{t}{\rightarrow} M$ is just
\[ g (g_1, g_2, \ldots , g_k)= (gg_1, g_2, \ldots , g_k).\]
\end{example}

\subsection{Classifying spaces and their cohomology}

In this subsection we recall the simplicial construction of classifying spaces and the Bott-Shulman-Stasheff model for their cohomologies. As references, we use \cite{BSS, Sea}.

{\bf Classifying spaces:}  We now recall the construction of the classifying space of a Lie groupoid, as the geometric realization
of its nerve. First of all, the nerve of $G$ is the following simplicial manifold. The manifold of $k$-simplices is
the $G_k$ previously described, with the simplicial structure given by the face maps:
\begin{equation*}
d_i(g_1,\dots,g_k)=
\begin{cases}
(g_2,\dots,g_k) & \text{if } i=0\\
(g_1,\dots,g_ig_{i+1},\dots,g_k) & \text{if } 0<i<k\\
(g_1,\dots,g_{k-1}) & \text{if } i=k\\
\end{cases}
\end{equation*}
and the degeneracy maps:
\begin{equation*}
s_i(g_1,\dots,g_k)=(g_1\dots,g_{i},1,g_{i+1},\dots,g_k)
\end{equation*}
for $0\leq i\leq k$.\\

The general constructions on simplicial
manifolds  \cite{Sea}  applied to our case produce a new topological space, the (thick) geometric realization,
and this is the classifying space $BG$ of $G$. Explicitly, $BG$ is the quotient space
\begin{equation*}
BG=(\coprod_{k\geq 0} G_k \times \Delta^k) / \sim
\end{equation*}
obtained by identifying $(d_i(p),v)\in G_k \times \Delta^k$ with
$(p,\delta_i(v))\in  G_{k+1} \times \Delta^{k+1}$ for any $p\in G_{k+1}$ and any
$v\in \Delta^k$. Here $\Delta^k$ denotes the
standard topological $k-simplex$ and $\delta_i: \Delta^k\rmap \Delta^{k+1}$ is the
inclusion on the $i$-th face.\\

{\bf The Bott-Shulman-Stasheff complex:} Next, we move the the cohomology of classifying spaces and recall the Bott-Shulman-Stasheff complex.
In general, $BG$ is
infinite dimensional and it is not a manifold.
However, there is a De-Rham theory that allows one to compute the
cohomology of $BG$. The Bott-Shulman complex, denoted
$\Omega(G_{\bullet})$, is the double complex
$$
\xymatrix{
 \vdots& \vdots& \vdots & \\
 \Omega^2(G_0)\ar[u]^{\derham}\ar[r]^{\udh} &\Omega^2(G_1)\ar[u]^{\derham}\ar[r]^{\udh}
 &\Omega^2(G_2)\ar[u]^{\derham}\ar[r]^{\udh} & \dots\\
 \Omega^1(G_0)\ar[u]^{\derham}\ar[r]^{\udh} &\Omega^1(G_1)\ar[u]^{\derham}\ar[r]^{\udh}
 &\Omega^1(G_2)\ar[u]^{\derham}\ar[r]^{\udh} & \dots\\
 \Omega^0(G_0)\ar[u]^{\derham}\ar[r]^{\udh} &\Omega^0(G_1)\ar[u]^{\derham}\ar[r]^{\udh}
  &\Omega^0(G_2)\ar[u]^{\derham}\ar[r]^{\udh} & \dots\\
}$$ where the vertical differential is just the De-Rham differential
and the horizontal differential $\udh$ is given by the simplicial
structure,
\begin{equation*}
\udh(\omega)=\sum_{i=0}^{p+1}(-1)^{i+p+q}d_i^*(\omega),
\end{equation*}
where $\omega \in \Omega^q(G_p)$.
The total complex of $\Omega(G_{\bullet})$ is the De Rham model for
the cohomology of $BG$. We will also need the normalized Bott-Shulman-Stasheff
complex, denoted $\hat{\Omega}(G_{\bullet})$, which is the
subcomplex of $\Omega(G_{\bullet})$ that consists of forms $\eta \in
\Omega^q(G_p)$ such that $s_i^*(\eta)=0$ for all $i=0,\dots ,p-1$.
The inclusion $\hat{\Omega}(G_{\bullet}) \to \Omega(G_{\bullet})$
induces an isomorphism in cohomology.

\begin{theorem}{\bf{(Dupont, Bott, Shulman, Stasheff, \dots)}} There is a natural isomorphism
\begin{equation*}
H(Tot(\Omega(G_{\bullet}))) \cong H(BG).
\end{equation*}
\end{theorem}

\begin{example}\rm \ 
When $G$ is a Lie group, one recovers (up to homotopy) the usual classifying
space of $G$. More generally, for the groupoid $G \ltimes M$ associated to an action 
of $G$ on $M$, $B(G \ltimes M)$ is a
model for the homotopy quotient
\begin{equation*}
B(G \ltimes M) \cong M_G= (EG\times M)/G,
\end{equation*}
hence the cohomology of $B(G\ltimes M)$ is isomorphic to the equivariant (Borel) cohomology of $M$.
\end{example}

{\bf Classical models and the Bott spectral sequence:}
The complex $\Omega(G_{\bullet})$ computes the cohomology of the classifying space of $G$ and is the place where 
various geometric structures live (e.g. the multiplicative two-forms of Poisson and even Dirac geometry). However
it is not optimal as a model for the cohomology of $BG$. Ideally, one would like to have a smaller infinitesimal model.

Let us first look at the case when $G$ is a Lie group. If $G$ is compact, a well-known theorem of Borel asserts
that
\[ H^*(BG)\cong S(\mathfrak{g}^*)^G.\]
For general Lie groups, Bott provides a spectral sequence which, when $G$ is compact, degenerates giving
the previous isomorphism. The Bott spectral sequence has the form:
\[ E^{p, q}_{1}= H^{p- q}(G; S^q\mathfrak{g}^*)) \Longrightarrow H^{p+q}(BG),\]
Here $H(G; V)$ denotes the {\it differentiable cohomology}  of $G$ (see below) with coefficients
in the representation $V$ of $G$. Also, $\mathfrak{g}^*$ is {\it the coadjoint representation} of $G$
and $S^*\mathfrak{g}^*$ is {\it the symmetric power of the coadjoint representation}. We remark at this point
that  while the notion of representation and that of differentiable 
cohomology extends without any problem to the context of Lie groupoids, the situation is quite different when it comes to the adjoint representation and taking its symmetric powers. These problems will be solved in this paper.

Regarding a better understanding of the cohomology of classifying spaces, another 
case that is worth mentioning here is that of the groupoid arising from an action of a Lie group $G$ on a manifold $M$.
As mentioned above, the resulting cohomology is the equivariant cohomology $H_{G}(M)$. Again, when $M$
is compact, the situation is particularly nice: one can find a rather small infinitesimal model consisting on ``equivariant
differential forms'' (which depends on the Lie algebra of $G$ rather than on $G$ itself) known as the Cartan model for equivariant 
cohomology. For noncompact groups a different model was found by Getzler \cite{Get}, which is a combination of Cartan's model and the model
computing differentiable cohomology of $G$ - we will refer to it as the Cartan-Getzler model.
In particular, this construction induces a spectral sequence very similar to the Bott spectral sequence which in the compact case
shows that the Cartan model does compute equivariant cohomology. 
We will see that such spectral sequences are a particular case
of a Bott spectral sequence for Lie groupoids.\\

\subsection{Representations and differentiable cohomology}

In this subsection we review the notion of representation of Lie groupoids and associated (differentiable) cohomology,
from the point of view of differential graded algebra. It will be this point of view that will be taken for introducing the
notion of representations up to homotopy. 

{\bf Differentiable cohomology:} For a Lie groupoid $G$ over $M$, the differentiable cohomology of $G$,
denoted $H^{*}_{\textrm{diff}}(G)$, or just $H^{*}(G)$, is the cohomology of the first line in the Bott-Shulman complex.
We denote by $C^{\bullet}(G)$ this first line. Hence $C^k(G)$ consists of smooth functions defined on $G_k$
while its differential, denoted
\[ d: C^k(G)\rmap C^{k+1}(G),\]
is the alternating sum of the $d_{i}^{*}$'s. The space $C^{\bullet}(G)$ has a natural algebra structure given by
\[ (f\star g)(g_1, \ldots , g_{k+p})= (-1)^{kp}f(g_1, \ldots , g_k) h(g_{k+1}, \ldots , g_{k+p}),\]
for $f\in C^k(G)$, $h\in C^p(G)$. In this way, $C^{\bullet}(G)$ becomes a DGA (differential graded algebra),
i.e. the differential $d$ and the product $\star$ satisfy the derivation identity
\[ d(f\star h)= d(f)\star h+ (-1)^k f\star d(h).\]
In particular, $H^{\bullet}(G)$ inherits a graded algebra structure.

\begin{example}\rm \ When $G$ is a Lie group, one recovers the group cohomology of $G$ with smooth cochains.
When $G$ is compact, using the Haar measure one can easily see that this cohomology vanishes in positive degree. 
This vanishing result actually holds for all proper groupoids (i.e. Lie groupoids $G$ over $M$ with the property that the map
$(s, t): G\rmap M\times M$ is proper). In particular, it also holds for the pair groupoid
$M\times M$ over $M$. The vanishing for the pair groupoid can be seen as a consequence of another important
property of differentiable cohomology: its Morita invariance (see \cite{Cra2}).

For a Lie group $G$, $H^{\bullet}(G)$ is related to the Lie algebra cohomology $H^{\bullet}(\mathfrak{g})$
by the so called Van Est map. Here $\mathfrak{g}$ is the Lie algebra of $G$. The Van Est map is known to be 
an isomorphism up to degree $k$ provided $G$ is $k$-connected. This relationship can actually be extended 
to all Lie groupoids. If the Lie group $G$ is connected, $H^{\bullet}(G)$ is isomorphic to $H^{\bullet}(\mathfrak{g}, K)$- 
the cohomology of $\mathfrak{g}$ relative to the maximal compact subgroup $K$ of $G$. In particular,
$H^{\bullet}(G)$ is finite dimensional for all connected Lie groups.

Finally, for the groupoid $G \ltimes M$ associated to an action 
of a Lie group $G$ on $M$, $H^{\bullet}(G \ltimes M)$ is the differentiable cohomology of
$G$ with coefficients in $C^{\infty}(M)$. 
\end{example}

{\bf Actions on vector bundles:}  Given a Lie groupoid $G$ over $M$, a quasi-action of $G$ on a vector bundle
$E$ over $M$ is an operation which associates to any arrow $g: x\rmap y$ in $G$ a linear map
\[ \lambda_g: E_x\rmap E_y\]
which varies smoothly with respect to $g$ (see also below). The quasi-action is called unital if $\lambda_g$ is the identity map
whenever $g$ is a unit. Finally, the quasi-action is called an action if the action condition
\[ \lambda_g\circ \lambda_h= \lambda_{gh}\]
holds for all $g, h\in G$ composable.

Alternatively, for a vector bundle $E$ over $M$, one can talk about the general linear groupoid of $E$, denoted $GL(E)$.
It is a groupoid over $M$ and its arrows from $x\in M$ to $y\in M$ are the linear isomorphisms
from $E_x$ to $E_y$. This space has a canonical smooth structure and, together with the composition of linear maps,
it becomes a Lie groupoid. With this, a quasi-action of $G$ on $M$ is a smooth map $\lambda: G\rmap GL(E)$ which commutes
with the source and the target maps. It is unital if it also commutes with the unit map. And it is an action if it is a
morphism of Lie groupoids.

There is  yet another, more algebraic way to view such actions. Given a vector bundle $E$ over $M$ we form the graded vector
space $C^{\bullet}(G; E)$ whose degree $k$ part is
\[ C^k(G; E):= \Gamma(G_k, t^*E).\]
This is a right graded module over the algebra $C^{\bullet}(G)$: given $\eta\in C^k(G; E)$, $f\in C^{k'}(G)$, their product $\eta\star f\in C^{k+k'}(G; E)$
is defined by 
\[ (\eta\star f)(g_1, \ldots , g_{k+k'})= (-1)^{k k'}\eta(g_1, \ldots , g_k) f(g_{k+1}, \ldots , g_{k+k'}).\]
One also considers the normalized subspace $\hat{C}^{\bullet}(G; E)$ of $C^{\bullet}(G;E)$ consisting of
those $\eta$ with the property that $s_{i}^{*}(\eta)= 0$ for all $i$.
In general, a quasi-action $\lambda$ of $G$ on $E$ induces a degree one operator
\[ D_{\lambda}: C^{\bullet}(G; E)\rmap C^{\bullet+ 1}(G; E)\]
by the formula 
\begin{eqnarray}
 D_{\lambda}(\eta)(g_1, \ldots , g_{k+1}) & = & (-1)^k\{\lambda_{g_1}\eta(g_2, \ldots , g_{k+1})+ \nonumber \\
                                                               &  + & \sum_{i= 1}^{k}(-1)^i \eta(g_1, \ldots, g_ig_{i+1}, \ldots , g_{k+1})+ (-1)^{k+1} \eta(g_1, \ldots, g_k)\}. \nonumber
\end{eqnarray}
Note that, in the particular case when $E$ is the trivial line bundle and $\lambda$ is the trivial action, this is precisely the
differential $d$ on $C^{\bullet}(G)$. With this, we have the following:

\begin{lemma}\label{lemma-quasi-action} The construction $\lambda\mapsto D_{\lambda}$ induces a 1-1 correspondence between
quasi-actions $\lambda$ of $G$ on $E$ and degree $1$ operators $D_{\lambda}$ on $C^{\bullet}(G; E)$ satisfying the
Leibniz identity
\[ D_{\lambda}(\eta\star f)= D_{\lambda}(\eta)\star f+ (-1)^k \eta\star d(f),\]
for $\eta\in C^{k}(G; E)$, $f\in C^{k'}(G)$.
Moreover, 
\begin{enumerate}
\item $\lambda$ is unital if and only if $D_{\lambda}$ preserves
the normalized subspace $\hat{C}(G; E)$. 
\item $\lambda$ is an action if and only if it is unital and $D_{\lambda}^2= 0$ (i.e. $C^{\bullet}(G; E)$ becomes
a right differential graded module over $C^{\bullet}(E)$).
\end{enumerate}
\end{lemma}

\begin{proof} The quasi-action can be recovered from $D_{\lambda}$ as follows: for $g: x\rmap y$ in $G$,
$v\in E_x$, choose $\alpha\in \Gamma(E)$ such that $\alpha(x)= v$ and put
\[ \lambda_g(v)= \alpha(y)+ d_{\lambda}(\alpha)(g).\]
This does not depend on the choice of $\alpha$. Indeed, if $\beta$ is another section with this property, we
may assume that $\beta- \alpha= f \gamma$ where $f\in C^{\infty}(M)$ vanishes at $x$ and $\gamma\in \Gamma(E)$,
and then one uses the Leibniz identity. The statement about unitality is immediate. The last part follows by a straightforward computation.
\end{proof}

\begin{definition} Let $G$ be a Lie groupoid over $M$. A representation of $G$ is a vector bundle $E$ over $M$ together with an action
$\lambda$ of $G$ on $E$. Given such a representation, the differentiable cohomology of $G$ with coefficients in $E$,
denoted $H^{\bullet}(G; E)$ is the cohomology of the complex $(C^{\bullet}(G; E), d_{\lambda})$. 
\end{definition}

\subsection{Connections and basic curvatures on groupoids}

In this subsection we introduce the notion of Ehresmann connections on Lie groupoids and their
associated (basic) curvatures. The choice of such a connection will be needed in order to construct
the adjoint representation of a Lie groupoid (a representation
up to homotopy); up to isomorphism, the adjoint representation is be independent of this choice.
The discussion is completely parallel to the one from the infinitesimal case
(Subsection 2.2 of \cite{AC1}). However, the presentation here is self-contained. One reason comes
from the fact that Ehresmann connections on Lie groupoids are of independent interest, we believe.
With further flatness conditions they have already appeared in \cite{BEH1, Tang}.

Throughout this subsection, $G$ is a Lie groupoid over $M$. Recall that
the unit map realizes $M$ as an embedded submanifold of $M$. Accordingly, 
for $x\in M$, the unit $1_x\in G$ at $x$ will be denoted by $x$. Similarly,
$T_xM$ will be viewed as a subspace of $T_xG$. Note that, for all $x\in M$,
\[ T_xM\oplus \textrm{Ker}(ds)_x= T_xG.\]
However, at arbitrary points $g\in G$, $\textrm{Ker}(ds)_x$ has no canonical 
complement. 

\begin{definition} An Ehresmann connection on $G$ is a sub-bundle $\mathcal{H}\subset TG$ which is complementary
to $(ds)$ and has the property that
\[ \mathcal{H}_{x}= T_xM, \ \ \forall\ x\in M.\]
\end{definition}

Here is an equivalent way of looking at connections which uses the vector bundle 
underlying the Lie algebroid of $G$:
\[ A= \textrm{Ker}(ds)|_{M}.\]
The construction of the Lie algebroid actually shows that there is  a short exact sequence
of vector bundles over $M$:
\[ t^*A\stackrel{r}{\rmap} TG\stackrel{ds}{\rmap} s^*TM,\]
where $r$ is given by right translations: for $g: x\rmap y$ in $G$, $r_{g}= (dR_g)_{1_y}: A_y\rmap T_g G$. 
With this, we see that the following structures are equivalent:
\begin{itemize}
\item An Ehresmann connection $\mathcal{H}$ on $G$.
\item  A right splitting of the previous sequence, i.e. a section $\sigma: s^*TM\rmap TG$ 
of $(ds)$, which restricts to the natural splitting at the identities.
\item A left splitting $\omega$ of the previous sequence which restricts to the natural one at the identities. Such a splitting can be viewed as a 1-form $\omega\in \Omega^1(G, t^*A)$ 
satisfying $\omega(r(\alpha))= \alpha$ for all $\alpha$.
\end{itemize}
All these are related by:
\[ \omega\circ r= \textrm{Id}, \ (ds)\circ \sigma= \textrm{Id}, r\circ \omega+ \sigma\circ (ds)= \textrm{Id}, \ \mathcal{H}= \textrm{Ker}(\sigma)= \textrm{Im}(\omega).\]

From now on, when talking about a connection on $G$, we will make no distinction between the sub-bundle $H$, 
the splitting $\sigma$ and the form $\omega$. In particular, we will also say that $\sigma$ is a connection on $G$.
Note that we will often say connection instead of Ehresmann connection.

\begin{lemma} 
Every Lie groupoid admits an Ehresmann connection.
\end{lemma}

\begin{proof}
 A standard partition of unity argument.
\end{proof}

\begin{remark}\rm \ Of course, there is a symmetric version of this definition which uses the target map instead of $s$.
The two are equivalent: any $\mathcal{H}$ as above induces a sub-bundle
\[ \overline{\mathcal{H}}:= (d\iota)(\mathcal{H})\]
which is complementary to $\textrm{Ker}(dt)$ and, at points $x\in M$, it coincides with the image of $T_{x}M$.
Similarly,  there is  a short exact sequence
of vector bundles over $M$:
\[ s^*A\stackrel{l}{\rmap} TG\stackrel{dt}{\rmap} t^*TM\]
where $l$ is given by left translations, and a connection $\mathcal{H}$ is the same as a splitting
$(\overline{\omega}, \overline{\sigma})$ of this sequence. In terms of $(\omega, \sigma)$,
\[ \overline{\sigma}_{g}(v)= (d\iota)_{g^{-1}}(\sigma_{g^{-1}}(v)),\ \overline{\omega}_g(X)= \omega_{g^{-1}}(d\iota)_{g}(X).\]
\end{remark}

What a connection gives us is, first of all, quasi-actions of $G$ on $A$ and on $TM$. 

\begin{definition} Given a connection $\sigma$ on $G$, we define the quasi-actions 
of $G$ on $TM$ and $A$, both denoted $\lambda$, given by
\begin{eqnarray*}
 \lambda_g(X)&=& (dt)_g(\sigma_g(X)),\\
 \lambda_g(\alpha)&=& - \omega_g(l_g(\alpha)),
\end{eqnarray*}
where $X\in T_{s(g)}M$ and $\alpha\in A_{s(g)}$. 
\end{definition}

The next piece of data that a connection gives us is a curvature term. 
Thinking of $\mathcal{H}$ as an Ehresmann connection on the bundle $s: G\rmap M$, one can talk about its 
Ehresmann curvature. There is however another type of curvature, called the basic curvature, which is
more important for our purposes- it measures the failure of $\mathcal{H}$ being closed under the
multiplication induced from $G$. For the actual definition, we prefer to use $\sigma$. Given 
$(g, h)\in G_2$ that we illustrate as:
\[ y\stackrel{g}{\lmap} \stackrel{h}{\lmap} x\]
 and given $v\in T_xM$, the expression
\[ \sigma_{gh}(v)- (dm)_{g, h}(\sigma_g(\lambda_h(v)), \sigma_h(v)) \in T_{gh}G\]
is killed by $(ds)_{gh}$, so it belongs to the image of $r_{gh}$. We denote by
\[ K^{\textrm{bas}}_{\sigma}(g, h)v\in A_y\]
the resulting element. When we vary $g, h$ and $v$ we obtain a section of a vector bundle over $G_2$.

\begin{definition} Given a connection $\sigma$ on $G$, we define the basic curvature of $\sigma$ as the
resulting section
\[ K^{\textrm{bas}}_{\sigma}\in \Gamma(G_2; \textrm{Hom}(s^*TM, t^*A)).\]
A Cartan connection on $G$ is an Ehresmann connection with the property that its basic curvature vanishes.
\end{definition}

Note that the basic curvature could be defined more directly (but less intuitively) by the formula:
\[ K^{\textrm{bas}}_{\sigma}(g, h)v= - \omega_{gh}((dm)_{g, h}(\sigma_g(\lambda_h(v)), \sigma_h(v))).\]
In words, the basic curvature can be described as the failure of the ``multiplicativity of $\omega$''. 
Note that, in terms of the bundle $\mathcal{H}$, the main property of the curvature 
can be rephrased as follows. To state it, recall that $TG$ is naturally a groupoid over $TM$: its structure maps are just the differentials
of the structure maps of $G$. It follows that

\begin{lemma} Given a connection on $G$, the associated bundle $\mathcal{H}$ is a subgroupoid of $TG$ if and only
if the basic curvature vanishes.
\end{lemma}

\begin{remark}
The notion of a connection on a Lie groupoid already appeared in Behrend's paper \cite{BEH1} and in Tang's \cite{Tang}. However, our terminology is different from the one in those papers. For Behrend,  a connection is what we call a  Cartan
connection, and a flat connection is  a Cartan connection which is integrable as a distribution. In \cite{BEH1} the author generalizes Getzler's model for equivariant cohomology to the case of groupoids endowed with flat Cartan connections, and mentions that it may be 
possible to describe the spectral sequence using a less restrictive type of connection (see Remark 3.14 therein). Indeed, in the present
paper we explain the construction of such a spectral sequence for an arbitrary connection in terms of representations up to homotopy.
While connections (in the sense of the present paper) exist for arbitrary Lie groupoids,
the existence of a flat Cartan connection is very restrictive. 
Indeed, one can show that if $G$ is a Lie groupoid over a compact simply connected manifold which has simply connected source fibers, then the existence of a flat Cartan connection  implies that $G$ is the
action groupoid associated to the action of a Lie group on a manifold. 
In order to prove this, one observes that a connection $\sigma$ on $G$ induces a connection $\nabla$ on the vector bundle $A$ over $M$. Moreover, the fact the basic
curvature of $\sigma$ is zero implies that the basic curvature of $\nabla$ (as defined in \cite{AC1})  vanishes. 
Also, the flatness of $\sigma$ implies that $\nabla$ is flat.
As shown in  Proposition 2.12 of \cite{AC1}, these conditions guarantee that $A$ is the Lie algebroid associated to a Lie algebra action. By the uniqueness of the integration, one concludes that $G$ is the groupoid associated to the corresponding group action.
\end{remark}

The main properties of the quasi-actions and of the basic curvature are collected in the following lemma.
Note that it will be the notion of representation up to homotopy that will allow us to understand all these
properties in a unified, more conceptual manner: they say that the adjoint representation is well-defined
as a representation up to homotopy.

\begin{proposition}\label{adjoint-basic-formulas} \label{lemma equations for the adjoint}
Given a connection $\sigma$ on $G$, the following
relations are satisfied:
\begin{enumerate}
\item The anchor $\rho: A\rmap TM$ is equivariant, i.e.
\begin{equation}\label{equation map of complexes}
\rho(\lambda_g(\alpha))= \lambda_g(\rho(\alpha)),
\end{equation}
for all $g\in G$ and $\alpha\in A_{s(g)}$.
\item For all
\[ z\stackrel{g}{\lmap} y \stackrel{h}{\lmap} x,\ X\in T_xM, \ \alpha\in A_x,\]
one has:
\begin{equation}\label{equation curvature X}
 \lambda_{g}\lambda_{h}(X)-\lambda_{gh}(X) =
\rho(K^{\textrm{bas}}_{\sigma}(g, h)(X)),
\end{equation}
\begin{equation}\label{equation curvature alpha}
\lambda_{g}\lambda_{h}(\alpha)- \lambda_{gh}(\alpha)=
K^{\textrm{bas}}_{\sigma}(g, h)(\rho(\alpha)).
\end{equation}
\item For all
\[ z\stackrel{g}{\lmap} y \stackrel{h}{\lmap} x \stackrel{k}{\lmap} w,\]
the basic curvature satisfies the cocycle equation:
\begin{equation}\label{equation cocycle curvature}
\lambda_{g}K^{\textrm{bas}}_{\sigma}(h, k)-
K^{\textrm{bas}}_{\sigma}(gh, k)+ K^{\textrm{bas}}_{\sigma}(g, hk)-
K^{\textrm{bas}}_{\sigma}(g, h)\lambda_k= 0.
\end{equation}
\end{enumerate}
\end{proposition}

\begin{proof}
For equation (\ref{equation map of complexes}) we compute:
\begin{align*}
\rho(\lambda_{g}(\alpha)) &=-d t_g \circ r_g \circ \omega_g \circ l_g (\alpha)\\
&= -d t_g  \circ l_g(\alpha)+d t_g \circ \sigma_g
\circ ds_g \circ l_g(\alpha)\\
&=\lambda_{g}(\rho(\alpha)).
\end{align*}
Next, we prove equation (\ref{equation curvature X}):
\begin{align*}
\rho(K^{\textrm{bas}}_{\sigma}(g, h)(X))&=d t_{gh} \circ l_{gh}
\circ \omega_{gh} \circ
d m_{g,h}(\sigma_g(\lambda_{h}(X)),\sigma_{h}(X))\\
&=d t_{gh} d m_{g,h}(\sigma_g(\lambda_{h}(X)),\sigma_h(X))\\
&\quad -d t_{h,g}\circ \sigma_{gh} \circ ds_{gh}\circ  d m_{g,h}(\sigma_g(\lambda_{h}(X)),\sigma_h(X))\\
&= d t_{gh} \sigma_g(\lambda_{h}(X))-d t_{gh} \sigma_{gh}(X)\\
&=\lambda_{g}(\lambda_{h}(X))-\lambda_{gh}(X).
\end{align*}

In order to prove equation (\ref{equation curvature alpha}) we will
use the fact that
\begin{equation}\label{easy sigma and D^0}
\sigma_g(\rho(\alpha))=l_g(\alpha)+r_g(\lambda_{g}(\alpha)).
\end{equation}
With this in mind, we can compute:
\begin{align*}
K^{\textrm{bas}}_{\sigma}(g, h)(\rho(\alpha))&=\omega_{gh}(d m_{g,h}(\sigma_g(\lambda_{h}(\rho(\alpha))),\sigma_h(\rho(\alpha))))\\
&=\omega_{gh}(d m_{g,h}(\sigma_{g}(\rho(\lambda_{h}(\alpha))),\sigma_h(\rho(\alpha))))\\
&=\omega_{gh}(d m{g,h}(r_g(\lambda_{g} \lambda_{h}(\alpha)),0))\\
&\quad +\omega_{gh}(d m_{g,h}(0,l_{h}(\alpha)))\\
&\quad +\omega_{gh}(d m_{g,h}(l_g(\lambda_{h}(\alpha)),r_h(\lambda_{h}(\alpha))))\\
&=\lambda_{g} \lambda_{h}(\alpha)+\omega_{g,h} l_{gh}(\alpha)=
\lambda_{g} \lambda_{h}(\alpha)-\lambda_{gh}(\alpha).
\end{align*}
In the computation we used the fact that
\begin{equation}\label{product is zero}
dm_{g,h}(l_g(\beta),r_h(\beta))=0, \text{   } \forall \beta \in A_y,
\end{equation}
as one easily verifies.

Finally, we will prove equation (\ref{equation cocycle curvature}).
Using equations $(\ref{equation curvature X})$ and $(\ref{easy sigma
and D^0})$ one can easily show that:
\begin{eqnarray*}
K^{\textrm{bas}}_{\sigma}(g,hk)(X)+\lambda_{g}\circ
K^{\textrm{bas}}_{\sigma}(h,k)(X),
\end{eqnarray*}
is equal to the expression
\begin{equation*}
\omega_{ghk}\circ d m_{g,hk}(\sigma_g(\lambda_{h}\circ
\lambda_{k})(X)-l_g(K^{\textrm{bas}}_{\sigma}(h,k)(X)),\sigma(hk)(X))(\clubsuit).
\end{equation*}
Next, we can use equation $(\ref{product is zero})$ to prove
that:
\begin{equation*}
(\clubsuit)=\omega_{ghk}\circ d m_{g,hk}(\sigma_g(\lambda_{h}\circ
\lambda_{k}(X)),d
m_{h,k}(\sigma_h(\lambda_{k}(X)),\sigma_k(X)))(\maltese).
\end{equation*}
Then, using the associativity of the multiplication one can show
that:
\begin{equation*}
(\maltese)=K^{\textrm{bas}}_{\sigma}(gh,k)(X)+K^{\textrm{bas}}_{\sigma}(g,h)(\lambda_{k}(X)).
\end{equation*}
This completes the proof.
\end{proof}

\section{Representations up to homotopy}

\subsection{The category of representations up to homotopy}

In this section we introduce the notion of representations up to homotopy of Lie groupoids. 
We start with the most compact definition. Given a Lie groupoid $G$ over $M$, we will consider graded vector bundle
$E= \bigoplus_{l\in \mathbb{Z}} E^{l}$ over $M$,  which are bounded both from above as well as from below (bounded, on short).
The space of $E$-valued cochains on $G$ will be considered with the
total grading:
\[ \uC(G; E)^n= \bigoplus_{k+l= n} C^k(G; E^l).\]
Given $\eta\in C^k(G; E^l)$, we say that $\eta$ is of bidegree $(k, l)$ and we denote by $|\eta|= k+ l$
its total degree; we also say that $k$ is the cocycle degree of $\eta$.
As in the ungraded case, and by the same formulas, $\uC(G; E)$ is a right $C(G)$-module.

\begin{definition}\label{def-rep-up}
A representation up to homotopy of $G$ on a (bounded) graded vector bundle $E$
over $M$ is a linear degree one operator 
\[ D:\uC(G;E)^{\bullet}
\rightarrow \uC(G;E)^{\bullet}, \]
called the structure operator of the representation up to homotopy $E$, satisfying $D^2= 0$ and the 
Leibniz identity
\[ D(\eta \star f )= D(\eta)\star f+(-1)^k \eta \star  \delta(f)\]
for $\eta \in C(G;E)^k$ and  $f \in C^{\bullet}(G)$.

A morphism $\Phi: E\rmap E'$ between two 
representations up to homotopy $E$ and $E'$ is a degree zero
$\uC^{\bullet}(G)$-linear map
\begin{equation*}
\Phi:\uC(G;E)^{\bullet}\rightarrow \uC(G;E')^{\bullet+ 1},
\end{equation*}
that commutes with the structure operators of $E$ and $E'$.  

We denote by 
$\uRRep^{\infty}(G)$ the resulting category and by $\uRep^{\infty}(G)$ the set of isomorphism
classes of representations up to homotopy. 
\end{definition}

To make this more explicit, we consider the bigraded vector space 
$C_G(\uEnd(E))$ which, in 
bidegree $(k,l)$, is
\begin{equation*}
C^k_G(\uEnd^l(E))=\Gamma(G_k,\uHom(s^*(E^{\bullet}),t^*(E^{\bullet+l}))).
\end{equation*}
As before, $s$ and $t$ are the maps $s(g_1,\dots,g_k)=s(g_k)$ and
$t(g_1,\dots,g_k)=t(g_1)$. Similarly we consider the bigraded space $C_G( \uHom(E, F))$
for any two vector bundle $E$ and $F$.

\begin{proposition}\label{decomposition with respect to F}
There is a bijective correspondence between representations up to
homotopy of $G$ on the graded vector bundle $E$ and sequences
$\{R_k\}_{k\geq 0}$ of elements $R_k\in C^{k}(G; \uEnd^{1-k}(E))$
which, for all $k\geq 0$, satisfy:
\begin{equation}\label{structure equations}
\sum_{j=1}^{k-1} (-1)^{j} R_{k-1}(g_1, \ldots , g_jg_{j+1}, \ldots, g_k)  = 
\sum_{j=0}^{k} (-1)^{j} R_{j}(g_1, \ldots , g_j) \circ R_{k-j}(g_{j+1}, \ldots , g_k). 
\end{equation}

Given two representations up to homotopy $E$ and $E'$, with corresponding sequences
$\{R_k\}$ and $\{R_{k}^{'}\}$, respectively, there is a bijective correspondence between 
morphisms $\Phi: E\rmap F$ of representations up to homotopy and sequences
$\{\Phi_k\}_{k\geq 0}$ of elements $\Phi_k\in \uC^k_G(\uHom^{-k}(E,E'))$
which, for all $k\geq 0$, satisfy
\begin{eqnarray}\label{equations for a map}
\sum_{i+j=k}(-1)^{j}\Phi_j(g_1,\dots,g_j) &\circ &
R_{i}(g_{j+1},\dots,g_k)=\\&=&\sum_{i+j=k}R_{j}^{'}(g_1,\dots , g_j)
\circ \Phi_{i}(g_{j+1},\dots,g_k)\nonumber\\
&&+\sum_{j=1}^{k-1}(-1)^{j}\Psi_{k-1}(g_1,\dots,g_jg_{j+1},\dots
,g_k).\nonumber
\end{eqnarray}
\end{proposition}

Before looking at the proof, let us briefly discuss the statement.

\begin{remark} In order to get an intuitive interpretation of
representations up to homotopy, let us now look at the structure equations (\ref{structure equations}) for low values of $k$.
\begin{itemize}
\item Equation (\ref{structure equations}) for $k= 0$ says that $\partial:= R_0: E^{\bullet}\rmap E^{\bullet+1}$ makes $E$ into a cochain
complex of vector bundles.  
\item Interpreting $\lambda:= R_1$ as a graded quasi-action
of $G$ on $E$, (\ref{structure equations}) for $k= 1$ says that 
$\lambda_g \partial =\partial \lambda_g$, i.e. the quasi-action is by maps of cochain complexes.

\item For $k=2$, equation(\ref{structure equations})  is
\begin{equation*}
\lambda_{g_1} \circ \lambda_{g_2}-\lambda_{g_1g_2}=\partial \circ
R_2(g_1,g_2) +R_2(g_1,g_2) \circ \partial,
\end{equation*}
which says that the quasi-actions are not necessarily associative,
but the operator $R_2$ provides homotopies between the cochain maps
$\lambda_{g_1} \circ \lambda_{g_2}$ and $\lambda_{g_1g_2}$. The
higher order equations are further compatibility conditions between
the homotopies.
\end{itemize}
\end{remark}

\begin{example} {\bf (Usual representations)} 
Of course, any ordinary representation $E$ of $G$ is an example of a
representation up to homotopy concentrated in degree zero.
\end{example}

\begin{example} {\bf (Cocycles)} 
Any closed cocycle $\eta \in C^k(G)$ with $k\geq 2$ induces a representation up to
homotopy structure in the complex which has the trivial line bundle
$\mathbb{R}$ in degrees $0$ and $k-1$, with zero differential. The
structure operators are $R_1=\lambda$ and $R_k= \eta$, where
$\lambda$ denotes the trivial representation of $G$ in $\mathbb{R}$.
The isomorphism class of this representation depends only on the
cohomology class of $\eta$ in $\uHdiff ^k(G)$.
\end{example}

\begin{definition} We will say that a representation up to homotopy is
weakly unital if $R_1(g)= \textrm{Id}$ whenever $g\in G$ is a unit. We say that it is unital
if, further, $R_k(q_1, \ldots, g_k)=0$ whenever $k\geq 2$ and one of the entries is a unit. 
\end{definition}

\begin{remark}\label{variations} {\bf (Variations)} 
If we want to view $\uRep^{\infty}(G)$ as an invariant of the groupoid, then there are some natural variations
to consider. First of all, it would be natural to require unitality. All the examples we have in  
mind are unital except the ones coming from cocycles (the last class of examples above). Nevertheless, these
are weakly unital and, by a normalization process, they can be made unital. Actually, one can view the normalization
process as a correspondence which associates to a weakly unital representation up to homotopy $E$ a unital one,
isomorphic to $E$ itself.

Another natural variation (which may be needed for better functorial properties) is obtained by allowing unbounded
complexes.  In that case, one should take Proposition  \ref{decomposition with respect to F} as the definition of representations up to homotopy
and morphisms between them. It is still possible to have a compact description, similar to Definition \ref{def-rep-up}, but one
needs some care. More precisely, $C(G, E)$ comes with a filtration by the cocycle degree
\[ F_p\uC(G; E)^n= \bigoplus_{k+l= n, k\geq p} C^k(G; E^l),\]
and one has to replace $C(G, E)$ by its completion with respect to this filtration, denote it $\bar{C}(G, E)$. This is just
\[ \bar{C}(G, E)^n= \Pi _{k+l= n} C^k(G; E^l),\]
whose elements should be thought of as infinite sums $\sum_{k\geq 0} \eta_k$ (with the index indicating the cocycle degree).
Note that $\bar{C}(G, E)$ is itself a right $C(G)$-module and it inherits a filtration $F_p$ (similar to the previous one, but
using products instead of sums). With these, one has to replace $C(G, E)$ by $\bar{C}(G, E)$ in Definition \ref{def-rep-up}
and require all the operators to be continuous.
\end{remark}

\begin{remark}\label{on-the-structure} {\bf (The derived category)} 
Still with the mind at $\uRep^{\infty}(G)$ as being an invariant of the groupoid, it is natural (and necessary- see e.g.
our comments below on tensor products) to consider the richer structure hidden behind the category
$\uRRep^{\infty}(G)$. For instance, making use of our DG approach to representations up to homotopy,
$\uRRep^{\infty}(G)$ can immediately be made into a DG category $\underline{\uRRep}^{\infty}(G)$. 
Let $\textrm{Hom}$ denote the hom-spaces
in $\uRRep^{\infty}(G)$. Then $\underline{\uRRep}^{\infty}(G)$
has the same objects as $\uRRep^{\infty}(G)$, but new graded hom-spaces:
\[ \underline{\textrm{Hom}}^{\bullet}(E, F)= \bigoplus_{l} \underline{\textrm{Hom}}^{l}(E, F),\]
where $\underline{\textrm{Hom}}^{l}(E, F)$ consists of $C^{\bullet}(G)$-linear maps which rise the
total degree by $l$:
\[ \Phi: C(G, E)^{\bullet}\rmap C(G, F)^{\bullet+ l}.\]
That this is a DG category is just a reflection of the fact that the graded hom-spaces $\underline{\textrm{Hom}}^{\bullet}$ are cochain complexes,
and the fact that the composition
is defined at this level. 
Here, the differential $D$ on $\underline{\textrm{Hom}}^{\bullet}(E, F)$ is induced from the structure operators of $E$ and $F$ by
\[ D(\Phi)= D_{F}\circ \Phi- (-1)^{|\Phi|} \Phi\circ D_{E} .\]
The strict category associated to a DG category has the same objects, but as hom-sets
the space $Z^0\underline{\textrm{Hom}}^{\bullet}$ of elements in $\underline{\textrm{Hom}}^0$ which are closed (with respect to $D$). In our case we get
\[ Z^0\underline{\textrm{Hom}}(E, F)= \textrm{Hom}(E, F),\]
i.e. our category $\uRRep^{\infty}(G)$.

By considering $\underline{\uRRep}^{\infty}(G)$ instead of $\uRRep^{\infty}(G)$, we offer ourselves a slightly more general point of view which allows
us to talk about homotopies between morphisms, functors ``up to homotopy'', the associated 
homotopy category, etc. For instance, for two representations up to homotopy $E$ and $F$, the set of homotopy classes of maps
from $E$ to $F$ is defined by
\[ [E, F]:= H^{0}(\underline{\textrm{Hom}}^{\bullet}(E, F))= \textrm{Hom}(E, F)/\sim\]
where $\sim$ is the homotopy relation: two morphisms $\Phi$ and $\Psi$ are homotopic if there exists
a $C(G)$-linear map $H: C(G, E)\rmap C(G, E)$ lowering the total degree by $1$, such that $\Phi- \Psi= D_{F}H+ HD_{E}$.

\begin{definition}\label{dervd-def} We define the derived category of $G$, denoted $\mathcal{D}\textrm{er}(G)$ with objects are 
the representations up to homotopy of $G$ and with hom-sets the $[E, F]$'s defined above.
\end{definition}

Let us mention here that, from an algebraic point of view,  our notion of representation up to homotopy is related to the
standard DG and derived categories of (DG) modules over a (DG) algebra. What happens in our case is that, if we replace
$C^{\bullet}(G)$ by an arbitrary DGA $(A, \star, \delta)$, we are not interested in all DG modules, but only those 
which are, as graded right $A$-modules, those of type $\mathcal{M}= \mathcal{E}\otimes_{A^0} A$, where $\mathcal{E}$ is
a graded $A^{0}$-module which is finitely generated projective in each degree. This explains why we define the hom's in the
derived category using homotopy classes of maps- which is possible in general only for cofibrant objects; see \cite{Hovey, Keller}
and also our Proposition \ref{rem-derived} below. Of course, one could consider general DG modules over $C(G)$, but in such an
algebraic approach we loose the geometric content of the situation, we believe, and we certainly loose some of the 
important results (e.g. the vanishing theorem of subsection \ref{The vanishing theorem}).
\end{remark}

The remainder of this section is devoted to the proof of Proposition \ref{decomposition with respect to F}.
The main observation is that 
the space $\uC(G,E)$ is generated as a
$\uC^{\bullet}(G)$-module by $\Gamma(M,E)$- hence, from the Leibniz rule for $D$ (or the $C(G)$-linearity of $\Phi$), 
the maps are determined by what they do on $\Gamma(E)$. The rest is just computations which could be done directly
once the correspondence is made explicit. Instead,
we will discuss some of the structure underlying the computations, which makes them much more 
transparent. 

First of all, we point out that for any graded vector bundle
$E$ over $M$, $C_G( \uEnd(E))$ is actually a bigraded algebra
which can be identified with the endomorphism algebra of the
(right!) $C(G)$-module $C(G; E)$. First, for $T\in  C^k_G(\uEnd^l(E))$
and $T'\in C^{k'}_G(\uEnd^{l'}(E))$, $T\star T'$ is defined by
\begin{equation}\label{formula-*}
(T\star T')(g_1, \ldots , g_{k+k'})= (-1)^{k(k'+l')}T(g_1, \ldots , g_k) \circ T'(g_{k+1}, \ldots , g_{k+k'}).
\end{equation}
Also, any $T\in C^k_G( \uEnd^{l}(E))$ acts on $\eta\in C^{k'}(G; E^{l'})$ by a similar formula:
\[ (T\star \eta)(g_1, \ldots , g_{k+k'})= (-1)^{k(k'+l')}T(g_1, \ldots , g_k)(\eta(g_{k+1}, \ldots , g_{k+k'})).\]
It is now straightforward to check that, indeed, $C_G( \uEnd(E))$ can be identified with the endomorphism algebra of the
$C(G; E)$. 

\begin{lemma} The correspondence which associates to $T\in C_G( \uEnd(E))$ the left multiplication by $T$,
\[ L_{T}: \uC(G;E)\to \uC(G;E), \ \eta\mapsto T\star\eta ,\]
defines a 1-1 correspondence between
\begin{enumerate}
\item elements $T\in  \uC^k_G( \uEnd^l(E))$.
\item operators on $\uC(G; E)$ which rise the bigrading
by $(k,l)$ and which are $\uC(G)$-linear.
\end{enumerate}
\end{lemma}

Similarly, $C_G( \uHom(E, F))$ is identified by the space of $C(G)$-linear maps from $C(G;E)$ to $C(G; F)$,
for any two graded vector bundles $E$ and $F$. The resulting pairings
\begin{eqnarray}\label{various-pairings}
C_G( \uEnd(F))\times C_G( \uHom(E, F))\rmap C_G(\uHom(E, F)),\\
C_G( \uHom(E, F))\times C_G( \uEnd(F))\rmap C_G( \uHom(E, F)).
\end{eqnarray}
are still denoted by $*$ and are given by formulas similar to (\ref{formula-*}).

Next, we point out a simple variation of the notion of (quasi-)actions on graded vector bundles.
A quasi-action $\lambda$ on the graded vector bundle $E$ is said to be graded if all the operations $\lambda_g$ preserve the degree. 
Any such $\lambda$ defines an operator $\hat{D}_{\lambda}$ rising the bidegree by $(1,0)$ (hence the total degree by $1$): 
for $\eta\in C^{k}(G; E^l)$, 
\begin{eqnarray}
\hat{ D}_{\lambda}(\eta)(g_1, \ldots , g_{k+1}) & = & (-1)^{k+l}\{\lambda_{g_1}\eta(g_2, \ldots , g_{k+1})+ \nonumber \\
                                                               &  + & \sum_{i= 1}^{k}(-1)^i \eta(g_1, \ldots, g_ig_{i+1}, \ldots , g_{k+1})+ (-1)^{k+1} \eta(g_1, \ldots, g_k)\}. \nonumber
\end{eqnarray}
The sign is chosen so that the operator $\hat{D}_{\lambda}$ is a derivation of the (right) $C(G)$-module $C(G; E)$
with respect to the total degree. The graded version of Lemma \ref{lemma-quasi-action} follows now easily:

\begin{lemma} The construction $\lambda\mapsto \hat{D}_{\lambda}$ is a 1-1 correspondence between
graded quasi-actions $\lambda$ of $G$ on $E$ and degree $1$ operators $\hat{D}_{\lambda}$ on $C^{\bullet}(G; E)$ satisfying the
Leibniz identity
\[ \hat{D}_{\lambda}(\eta\star f)= \hat{D}_{\lambda}(\eta)\star f+ (-1)^{|\eta|} \eta\star d(f),\]
for $\eta\in C(G; E)$, $f\in C(G)$ (recall that $|\eta|$ is the total degree). 
\end{lemma}

Note that, in terms of the endomorphism algebra, a graded quasi-action is just an element of bidegree $(1, 0)$ 
\[ \lambda\in C^1(G, \uEnd^0(E)),\]
and the associated operator is just
\[ \hat{D}_{\lambda}= L_{\lambda}+ \hat{D}_0,\]
where $L_{\lambda}$ is multiplication by $\lambda$ and $\hat{D}_0$ is the operator corresponding to $\lambda= 0$. 
Next, note that the operator $\hat{D}_0$ induces an operator on $C_G(\uEnd(E))$: for $T\in C^k_G( \uEnd^l(E))$,
$\hat{D}_0(T)$ has bidegree $(k+1,l)$ and it is given by 
\[\hat{ D}_0(T)(g_1, \ldots , g_{k+1})  =  (-1)^{k+l} \sum_{i= 1}^{k}(-1)^i \eta(g_1, \ldots, g_ig_{i+1}, \ldots , g_{k+1}).\]
The fact that this is induced by the operator $\hat{D}_0$ on $C(G; E)$ is explained by the following straightforward lemma.

\begin{lemma} For any $T\in C_G( \uEnd(E))$ and any $\eta\in C(G;E)$,
\[ \hat{D}_0(T\star \eta)= \hat{D}_0(T)\star \eta+ (-1)^{|T|} T\star \hat{D}_0(\eta).\]
\end{lemma}

Note that, for any two graded vector bundles $E$ and $F$, there is a similar operator $\hat{D}_0$ on $C_G( \uHom(E, F))$
which satisfies the previous equations for all $T\in C_G( \uHom(E, F))$ and $\eta\in C(G; E)$. Similarly, the same equations
are satisfied with respect to the pairings (\ref{various-pairings}).

We now go back to Proposition \ref{decomposition with respect to F}. We start with the first part. Once we observe that $\uC(G;E)$ is generated as a
$\uC^{\bullet}(G)$-module by $\Gamma(M,E)$, the Leibniz identity for $D$ implies that 
$D$ decomposes as a sum 
\begin{equation}\label{decomposition}
D=D_0+D_1+ D_2+ \dots ,
\end{equation}
where:
\begin{itemize}
\item For each $i\geq 0$, $D_i$ comes is an operator  which ``rises'' the bidegree by $(i, 1-i)$. 
\item For $i\neq 1$, $D_i$ is $C(G)$-linear, hence it comes from left multiplication by an element
$R_i\in C^{i}_G( \uEnd^{1-i}(E))$.
\item For $i= 1$, $D_1$ is a derivation, hence it comes from a graded quasi-action $\lambda$,
also denoted $R_1$. 
\end{itemize}
In terms of the $R_i$'s, we have
\[ D(\eta)= \sum_{k\geq 0} R_k\star \eta+ \hat{D}_0(\eta).\]
The condition $D^2= 0$ can now be easily written out using the last lemma and it becomes
\[ \sum_{j, i} R_j\star R_i+ \sum_{i} \hat{D}_0(R_i) =0.\]
Looking at homogeneous degrees, the equation above becomes the set of equations
\[ \sum_{i+ j= k} R_j\star R_i+ \hat{D}_0(R_{k-1})= 0,\]
which are precisely the equations in the statement. The second part of the proposition is similar.
We obtain that $\Phi$ decomposes as 
\begin{equation*}
\Phi(\eta)= \Phi_0\star \eta+\Phi_1\star\eta+\Phi_2\star\eta+\dots,
\end{equation*}
where $\Phi_k\in \uC^k_G(\uHom^{-k}(E,E'))$. The fact that $\Phi$ commutes with the structure operators 
translates into:
\[\sum_{i, j} R_{j}^{'}\star \Phi_i+ \sum_j \hat{D}_0(\Phi_j)= \sum_{i, j} \Phi_j\star R_i .\]
Looking at homogeneous components, that means that for all $k$ one has
\[ \sum_{i+ j= k} R_{j}^{'}\star \Phi_i+ \hat{D}_0(\Phi_{k-1})= \sum_{i+ j= k} \Phi_j \star R_i.\]
Writing out $\star$ and $\hat{D}_0$, we obtain the equations from the second part of the proposition.

\subsection{Main example: the adjoint representation}

The aim of this subsection is to show that, associated to any Lie groupoid $G$, there is 
an ``adjoint representation'' which is a representation up to homotopy, well-defined up to isomorphism.
Again, this can be seen as a global version of the adjoint representation of Lie algebroids.
We refer the reader to \cite{AC1} for more explanations on the need of
graded vector bundles and on the choice of the underlying cochain complex (see Example 2.6. therein). 

We start with the underlying complex.

\begin{definition}
Given a Lie groupoid $G$ over $M$ with Lie algebroid $A$, the
adjoint complex of $G$ denoted $\uAd(G)$ is the complex of vector
bundles
\begin{equation*}
 \uAd(G):=A \stackrel{\rho}{\rmap} TM,
\end{equation*}
where $A$ has degree zero, $TM$ has degree one and $\rho$ is the
anchor map. We will write $\uAd$ instead of $\uAd(G)$ when the
groupoid is clear from the context.
\end{definition}

Next, to make $\uAd(G)$ into a representation up to homotopy, 
we first have to choose an Ehresmann connection $\sigma$ on $G$. 
As in our preliminary subsection on connections, we 
consider the induced quasi-actions $\lambda$ 
on $A$ and on $TM$ and  its basic curvature $K^{\textrm{bas}}_{\sigma}$. 
The equations that they satisfy (Proposition \ref{adjoint-basic-formulas}) can
now be interpreted as the structure equations of Proposition \ref{decomposition with respect to F}.
Using the language of representations up to homotopy, it is now easy to see that 
Proposition \ref{adjoint-basic-formulas} translates into the following.

\begin{proposition}\label{thm adjoint representation}
Let $G$ be a Lie groupoid and $\sigma$ a connection on $G$. Then:
the operators
\begin{equation*}
R_0=\rho,\hspace{2mm}  R_1=\lambda,\hspace{2mm}
R_2=K^{\textrm{bas}}_{\sigma},
\end{equation*}
give the adjoint complex $\uAd(G)$ the structure of a unital
representation up to homotopy.
\end{proposition}
\begin{proof}
In view of Proposition \ref{decomposition with respect to F}, in
order to prove the first statement we have to check a set of equations on the structure operators.
It is easy to see that these equations are precisely the ones that were proved to hold in Proposition \ref{adjoint-basic-formulas}.
\end{proof}
\begin{definition} Given a connection $\sigma$ on $G$, we denote by $\textrm{Ad}_{\sigma}(G)\in \uRRep^{\infty}(G)$ the
resulting representation up to homotopy of $G$.
\end{definition}

Next, we point out that, up to isomorphisms, i.e. for the element induced in $\uRep^{\infty}(G)$, the choice of the connection is irrelevant.

\begin{proposition} If $\sigma$ and $\sigma'$ are two connections on $G$, then the representations up to homotopy
$\uAd_{\sigma}(G)$ and $\uAd_{\sigma'}(G)$ are canonically isomorphic.
\end{proposition}

\begin{proof}
We drop $G$ from the notation. We will construct an isomorphism
$\Psi$ from $\uAd_{\sigma}$ to $\uAd_{\sigma'}$. The bundle map is
given by:
\begin{equation*}
\Psi_0(v)=(-1)^{|v|}v.
\end{equation*}
there is also $\Psi_1 \in C^1_G(\uEnd^{-1}(\uAd))$, defined by:
\begin{equation*}
(\Psi_1)_g(X)=\sigma_g(X)-\sigma'_g(X).
\end{equation*}
In order to show that this is a map of representations we need to
prove the equations:
\begin{eqnarray*}
(\Psi_1)_g \circ \rho -\Psi _0 \circ {\lambda}^{\sigma}_g&=&\rho
\circ (\Psi_1)_g -
{\lambda}^{\sigma'}_g \circ \Psi_0,\\
(\Psi_1)_g \circ \lambda^{\sigma}_h-\Psi_0 \circ
K^{\textrm{bas}}_{\sigma}(g,h)&=& K^{\textrm{bas}}_{\sigma'}(g,h)
\circ \Psi_0 - \lambda^{\sigma'}_g \circ (\Psi_1)_h + (\Psi_1)_{gh},
\end{eqnarray*}
which follow from computations similar to those in the proof of
Proposition \ref{lemma equations for the adjoint}. It is clear that $\Psi$
is an isomorphism with inverse given by $ \Psi_0- \Psi_1$.
\end{proof}

\begin{definition} We denote by 
\[ \mathrm{Ad}_{G}\in \uRep^{\infty}(G)\] 
the isomorphism class of the representations up to homotopy $\mathrm{Ad}_{\sigma}(G)$.
\end{definition}

\subsection{Operations}

In this subsection we discuss some of the basic operations on representations up to homotopy. The full discussion of
symmetric powers (and of tensor products) is more involved and it is deferred to \cite{ACD}; here we only give
a brief outline. Also, since symmetric powers appear in the construction of the Bott spectral sequence, we 
introduce axiomatically the notion of $q$-th symmetric power operation.

\begin{example} {\bf (pull-back's)}
The category of representations up to homotopy behaves in a
natural way with respect to groupoid morphisms: any morphism of Lie groupoids
$\varphi:H \rightarrow G$ induces a functor
\begin{equation}
\label{pb-fct} 
\varphi^*: \uRRep^{\infty}(G)\rmap \uRRep^{\infty}(H).
\end{equation}
At the level of objects, for a representation up homotopy
$E= (E, R_0, R_1, \ldots )$ of $G$, $\varphi^*(E)$ has its structure
determined by
\begin{equation*}
\varphi^*(R_k)(h_1,\dots ,h_k)=R_k(\varphi(h_1),\dots,\varphi(h_k)).
\end{equation*}
The categories of ordinary representations of Morita equivalent groupoids
are equivalent, and therefore, this category is an invariant of the associated
stack. We expect that there is a similar invariance result for the
category of representations up to homotopy -at least at the level of the derived category.
\end{example}

\begin{example}\label{dual representation} {\bf (duals)}
Let $G$ be a groupoid and $E= (E, R_0, R_1,\ldots)$ a representation up to homotopy of $G$.
Then the dual of $E$ is $E^*= (E^*, R_{0}^{*}, R_{1}^{*}, \ldots )$, where
\begin{equation*}
R_k^{*}(g_1,\dots,g_k)= (-1)^{k+1} \left(
R_k({g_k}^{-1},\dots,{g_1}^{-1})\right)^*,
\end{equation*}
is a new representation up to homotopy of $G$. We would like to warn the reader that, 
with the notation $R_0=\partial$, 
the cochain complex underlying $E^*$ has $R_{0}^{*}=  -\partial^*$ which is not the standard 
dual of the cochain complex $(E, \partial)$ (which has as boundary $(-1)^n\partial^*$ on $E_{n}^{*}$).
\end{example}

\begin{example}\label{coadjoint} {\bf (coadjoint representation)}
Applying the previous construction to the adjoint representation, we obtain the coadjoint representation
of $G$, again well-defined up to isomorphism, with a representative $\textrm{Ad}_{\sigma}^{*}$ for each Ehresmann
connection $\sigma$ on $G$.
\end{example}

\begin{example}\label{mapping cones} {\bf (mapping cones)} The construction of mapping cones \cite{Weibel}
applies also to our situation. Recall first, in the case of complexes (of vector bundles, say), the mapping cone of a map
of complexes $\Phi_0: (E, \partial)\rmap (F, \partial)$ is a new complex of vector bundles, $C(\Phi_0)$, with
\[ C(\Phi_0)^n= E^n\oplus F^{n-1},\ \ \partial(e, f)= (\partial(e), \Phi(e)- \partial(f)) .\]
Similarly, given a morphism $\Phi: E\rmap F$ in $\uRRep^{\infty}(G)$, one can construct the mapping 
cone $C(\Phi)\in \uRRep^{\infty}(G)$, whose underlying complex is the mapping cone of $\Phi_0$; 
the structure operators are 
\[ R_k(g_1, \ldots , g_k) (e, f)= (R_k(g_1, \ldots , g_k)e, \Phi_k(g_1, \ldots , g_k)e- (-1)^k R_k(g_1, \ldots , g_k)f).\]
Actually, as a DG $C(G)$-module, $C(\Phi)$ 
is just  the mapping complex of the map $\Phi: C(G;E)\rmap C(G; F)$.
\end{example}

\begin{example} {\bf (tensor products)}  Given $E, E'\in \uRep^{\infty}(G)$,
defining their tensor product $E\otimes E'\in \uRep^{\infty}(G)$ is more subtle.
Denoting the structure elements $\{R_{k}:k\geq 0\}$ for $E$ and $\{R_{k}^{'}:k\geq 0\}$ for $E'$,
the similar operators for $E\otimes E'$ are clear in small degrees:
\begin{enumerate}
\item The degree zero element should be $R_{0}\otimes \textrm{Id}+\textrm{Id}\otimes R_{0}^{'}$
(so that the underlying complex is the tensor product of the underlying complexes of $E$ and $E'$).
\item The degree one term, i.e. the quasi-action, should be the diagonal quasi-action.
\end{enumerate}
However, for higher $k$'s, the $R_{k}$ is more subtle, is of combinatorial nature and is the
subject of \cite{ACD}. Here is a short outline. First, the possible operations associated to representations
up to homotopy are encoded graphically by trees. The construction of the tensor product depends on an universal 
choice (i.e. not depending on $G$)- a formal sum $\omega$ involving trees which satisfy a certain universal Maurer-Cartan equation
($\omega$ is called an universal Maurer-Cartan element). The resulting tensor product operation $E\otimes_{\omega} F$ 
does not depend on the choice of $\omega$, up to isomorphisms of representations up to homotopy. Morphisms are treated
similarly. However, the resulting ``functor''
\[ \cdot \otimes\cdot : \uRRep^{\infty}(G)\times \uRRep^{\infty}(G)\rmap \uRRep^{\infty}(G)\]
is a functor only up to homotopy (i.e. it respects the composition of morphisms only up
to homotopy). Of course, this fits into the discussion of Remark \ref{on-the-structure},
with the conclusion that $\otimes$ should be seen as an $\infty$-functor at the
level of the DG categories $\underline{\uRRep}^{\infty}$.
\end{example}

\begin{example}\label{axiomat} {\bf (symmetric powers)}
A similar discussion applies to symmetric powers $\uS^q$ of representations up to homotopy ($q$ non-negative integer), which we
need here in order to discuss  Bott spectral sequences. We refer again to \cite{ACD}
for the general discussion of these operations (existence and uniqueness). For the purpose of Bott spectral sequences, 
we only need to know what $\uS^q$ does on representations up to homotopy (and not on general
morphisms) and its basic properties. This allows us to work with an axiomatic description,
with a set of axioms which are much weaker than the properties satisfied by the constructions of
\cite{ACD}. This axiomatic approach has the advantage that in specific situations,
one can use simple versions of the symmetric power functor. 

In short, one would like to have a ``natural extension'' of the 
the usual symmetric power operations
\[ S^{q}: \textrm{Ch}(M)\rmap \textrm{Ch}(M) \]
defined on the category $\textrm{Ch}(M)$ of complexes of vector bundles over $M$
(with maps of complexes as morphisms). To explain the axioms, we need the following definition. 

\begin{definition}  We say that a morphism $\zeta: E\rmap F$ between two representations up to homotopy
is strict if all the components $\zeta_k$ with $k\geq 1$ vanish. 
\end{definition}

In other words, a strict morphism is just a map of complexes $\zeta: E\rmap F$ (between the complexes 
underlying $E$ and $F$) with the property that it is also a morphism of representations up to
homotopy.

\begin{definition}\label{axiomati}
Let $q$ be a non-negative integer. A $q$-th symmetric power operation on representations up to homotopy
is an operation $\uS^q$ which associates to any $E\in \uRep^{\infty}(G)$ 
a new $\uS^q(E)\in \uRep^{\infty}(G)$, with the following properties:
\begin{enumerate}
\item[(S1)] for any $E\in \uRep^{\infty}(G)$, 
the complex underlying $\uS^q(E)$
is just the $q$-th symmetric power $S^qE$ of the complex underlying $E$.
\item[(S2)] for any $E\in \uRep^{\infty}(G)$, the quasi-action 
underlying $\uS^q(E)$ is the diagonal quasi-action on $S^qE$ induced by the quasi-action
underlying $E$.
\item[(S3)] for any strict morphism $\zeta: E\rmap F$, the usual $q$-th symmetric power of the chain map
$\zeta$, $S^q\zeta: S^qE\rmap S^qF$, is a strict morphism from $\uS^qE$ to $\uS^qF$.
\item[(S4)] for any groupoid morphism $\varphi: G\rmap H$, $\uS^q$ commutes with the pull-back functor (\ref{pb-fct}). 
\end{enumerate}
\end{definition}

The construction of \cite{ACD} defines $\uS^q$ also at the level of morphisms, with the following properties
replacing (S3):
\begin{enumerate}
\item[(S3')] for any $\Phi: E\rmap F$ in $\uRRep^{\infty}(G)$, the morphism of complexes
underlying $\uS^q\Phi$ 
is the $q$-th symmetric power of the morphism
of complexes 
underlying $\Phi$. 
\item[(S3'')] $\uS^q$ preserves strict morphisms.
\end{enumerate}

Again, as in the case of tensor products, the full outcome of \cite{ACD} is best
expressed in the language of DG categories (see Remark \ref{on-the-structure}), when $\uS^q$ becomes
an $\infty$-functor on the DG category $\underline{\uRRep}^{\infty}(G)$. To make $\uS^q$ into a true functor, the most natural thing to do is to pass
to the homotopy or derived category of $G$. Alternatively (but only for some purposes),
one can stay as simple as possible, and pass to 
the subcategory of $\uRRep^{\infty}(G)$ with the same objects but only strict morphisms allowed- and
that is what the previous definition does.
\end{example}

\subsection{Cohomology and the derived category}

In this subsection we discuss the (differentiable) cohomology with 
coefficients in representations up to homotopy and quasi-isomorphisms. 
First of all, our definition of representations up to homotopy clearly comes with an
associated cohomology theory.

\begin{definition} Given a representation up to homotopy $E$ of $G$, with structure operator $D$, we define the
differentiable cohomology of $G$ with coefficients in $E$, denoted $\uHdiff^{\bullet}(M; E)$,
as the cohomology of the complex $(C(G; E), D)$. 
\end{definition}

When it comes to cohomology (but not only then), the
relevant notion of equivalence is the following.

\begin{definition} A morphism $\Phi: E\rmap E'$ between two 
representations up to homotopy of a groupoid
$G$ is called a quasi-isomorphism if the map of vector bundles
$\Phi_0$ induces isomorphisms in cohomology at every point.
\end{definition}

Following the standard strategy (see e.g. \cite{Hovey, Keller}), the derived category 
of $G$, $\mathcal{D}\textrm{er}(G)$ should be defined at this point, as the category
obtained from $\uRRep^{\infty}(G)$ by formally inverting quasi-isomorphisms. What happens is
that, since our objects have a cofibrant behavior (see also our Remark \ref{on-the-structure}), 
in our setting quasi-isomorphisms are the same thing as homotopy equivalences, i.e. 
$\Phi$'s which become invertible when passing to our derived category $\mathcal{D}\textrm{er}(G)$ 
(Definition \ref{dervd-def}). 

Instead of making a reference to the various (rather complicated) descriptions of cofibrant
DG modules, let us sketch a self-contained proof of the following:

\begin{proposition}\label{rem-derived} Any quasi-isomorphism $\Phi: E\rmap F$ between two representations up
to homotopy is a homotopy equivalences.
\end{proposition}

\begin{proof} First two observations:
\begin{itemize}
\item If $(E, \partial)$ is a complex of vector bundles which is acyclic
(the cohomology bundle is trivial), then it is contractible: we find $h: E{\bullet}\rmap E^{\bullet -1}$
such that $h\partial+ \partial h+ \textrm{Id}= 0$. Moreover, one can choose $h$ such that $h^2= 0$. 
This is a standard result. One can e.g. choose a metric on $E$, consider the resulting adjoint $\partial^*$ of $\partial$ 
and then choose $h= \Delta^{-1}\partial^*$, where $\Delta= \partial\partial^*+ \partial^*\partial$.
\item  2: If $(E, D)$ is a representation up to homotopy whose underlying complex
$(E, \partial)$ is acyclic, then it is contractible, i.e. one can find $H\in \underline{\textrm{Hom}}^{-1}(E, E)$
such that $HD+ DH+ \textrm{Id}= 0$. To see this, consider $\delta_{D}= D- \partial$
acting on $C(G; E)$ and put
\[ H:= h(1+ (\delta_{D}h)+  (\delta_{D}h)^2+ \ldots ) : C(G; E)\rmap C(G; E) .\]
When applied to each element, the sum is finite hence it makes sense (it is just the inverse of
$(1- \delta_Dh)$). A formal computation shows that $H$ satisfies the required equation. The
fact that $h^2= 0$ implies that $H$ is $C(G)$-linear.
\end{itemize}
If $\Phi: E\rmap F$ is a quasi-isomorphism between representations up to homotopy, we first pass to the mapping
cone $C(\Phi)$ (see Example \ref{mapping cones}). Its underlying complex of vector bundles, $C(\Phi_0)$,
is acyclic because $\Phi_0$ is a quasi-isomorphism. From the first observation, it is contractible. From the second observation,
$C(\Phi)$ is contractible. Writing the resulting operator $H$ in a matrix form
\[ H=  \left(
\begin{array}{cc}
                h_1 & \bullet \\
                \Psi & h_2\,
        \end{array}
   \right) : C(G; E)\oplus C(G; F)\rmap C(G; E)\oplus C(G; F),\]
all the entries are $C(G)$-linear. Writing out the contracting homotopy equation for $H$, we find that $h_1$ provides
a homotopy between $\Psi\Phi$ and the identity, and similarly $h_2$ for $\Phi\Psi$. This shows that, at the level of the derived category,
$\Phi$ is invertible with inverse $\Phi$. 
\end{proof}

\begin{remark} As most of our results, the last proposition also holds if we allow unbounded
complexes (see Remark \ref{variations}). Note for instance that the sum above defining $H$ 
still makes
sense due to the fact that we use completions and all the operators are continuous.
\end{remark}

The next result shows that the differentiable cohomology functor descends to the derived category.
Given a representation $E\in \uRep^{\infty}(G)$, there is a decreasing filtration
on the complex $C(G; E)$:

\begin{equation}\label{filtration}
\cdots \subseteq L^2(C(G;E))\subseteq L^1(C(G;E))\subseteq
L^0(C(G;E))=C(G;E),
\end{equation}
where
\begin{equation*}
L^k(C(G;E))=C^k(G;E)\oplus C^{k+1}(G;E)\oplus C^{k+2}(G;E)\oplus
\cdots.
\end{equation*}
This filtration gives a spectral sequence $\mathcal{E}_r^{p,q}(E)$
converging to $\uHdiff^{p+q}(G;E)$.

\begin{proposition}
 If $\Phi: E\rmap E'$ is a quasi-isomorphism between two representations up to homotopy of $G$, then the map induced in cohomology
\begin{equation*}
\Phi:\uHdiff(G; E)\rightarrow \uHdiff(G; E'),
\end{equation*}
is an isomorphism.
 \end{proposition}

\begin{proof}
$\Phi$ induces a map of spectral sequences
$\Phi:({\cal{E}}_r^{p,q}(E),d_r^{p,q})\rightarrow
({\cal{E}}_r^{p,q}(E'),d_r^{p,q})$ which is an isomorphism for $r=1$.
The result follows from the comparison theorem of spectral sequences
\cite{Weibel}.
\end{proof}

The spectral sequence simplifies when the representations are
regular in the following sense:

\begin{definition} A complex of vector bundles $(E,\partial)$ is called 
regular if $\partial$ has constant rank. In this case we denote by 
${\cal{H}}^{\bullet}(E)$ the cohomology of $(E,\partial)$, which is a graded vector bundle
due to the constant rank condition.
\end{definition}

\begin{theorem}\label{spectral}\label{push a representation to the cohomology}
Let $E$ be a unital representation up to homotopy of a Lie groupoid
$G$ whose underlying complex is regular. Then each cohomology vector bundle
${\cal{H}}^{\bullet}(E)$ has the structure of ordinary representation of $G$ with the action
defined by
\begin{equation*}
\lambda_{g}(\overline{v}):=\overline{\lambda_{g}(v)},
\end{equation*}
and there is a spectral sequence:
\begin{equation*}
\mathcal{E}_2^{p,q}(E)\cong \uHdiff^p(G;{\cal{H}}^q(E))\Rightarrow
\uHdiff^{p+q}(G;E).
\end{equation*}
Moreover, the action on the ${\cal{H}}^{\bullet}(E)$'s  can be extended to a representation up to
homotopy structure in the cohomology complex ${\cal{H}}(E)$, with
zero differential, for which there is a quasi-isomorphism:
\begin{equation*}
\Phi:C(G; E)\rightarrow C(G;{\cal{H}}(E)).
\end{equation*}
\end{theorem}
\begin{proof}
The fact that the formula
$\lambda_{g}(\overline{v})=\overline{\lambda_{g}(v)}$ defines a
representation structure on ${\cal{H}}^q(E)$ follows from the structure
equations (\ref{structure equations}) and the fact that the representation is unital.
On the other hand, one easily shows that
$\mathcal{E}_1^{p,q}(E) \cong C^p(G,{\cal{H}}^q(E))$ and that the
differential $d_1^{p,q}$ is the one given by the representation
structure. The proof of the last statement uses homological
perturbation theory and can be copied word by word from the that of
the infinitesimal version, which can be found in Section $3.3$  of
\cite{AC1}. See also Kadeishvili's paper \cite{Kad}.
\end{proof}


\begin{remark}\label{reasons-complexes} The last theorem shows that, at the level of the derived category,
the regular representations up to homotopy can be replaced by new ones with zero-differential.
The last ones correspond to graded ordinary representations together with 
certain structure cocycles $\{R_{k}: k\geq 2\}$ (true cocycles!). Hence one may think that
the passing from ordinary representations to $\uRRep^{\infty}(G)$ is about:
\begin{itemize}
\item allowing graded objects
together with such cocycles. 
\end{itemize}
However, we would like to point out that this is only one of the
features of $\uRRep^{\infty}(G)$; the most important one is different: 
\begin{itemize}
\item allowing ``non-regular representations''.
\end{itemize}
Morally, these hide ``singular representations'' (in the sense that the associated cohomology bundles
are no longer vector bundles). Probably the adjoint representation is the best examples in which both features
can be seen: it is non-regular in general and, even when it is regular, the 2-cocycles involved (the basic curvature) may
be non-trivial. The infinitesimal version of this remarks appears in \cite{AC1} (see e.g. Example 2.6 therein)
and shows that the full structure of representations up to homotopy is needed in order to make sense of the adjoint 
representation). 
\end{remark}

\subsection{The vanishing theorem}
\label{The vanishing theorem}

One of the main properties of differentiable cohomology with coefficients in ordinary representations
is the fact that, for proper groupoids (the generalization of compactness when going from groups to groupoids),
it vanishes in all positive degrees. This result turns out to be very useful when dealing with deformations
or rigidity phenomena (see e.g. \cite{regular, linear} for such applications). With such applications in mind, one can try
to generalize this result to general representations up to homotopy. This is interesting also from the point of
view of Bott's spectral sequence: it is precisely this kind of
phenomenon that allows one to obtain the Cartan model from that of
Getzler.

Recall first that a Lie groupoid $G$ over $M$ is called proper if the map
\begin{equation*}
(s,t):G\rightarrow M\times M
\end{equation*}
is proper. The vanishing theorem for ordinary representations \cite{Cra2} says that,
if $E$ is an ordinary representation of a proper groupoid, then $\uHdiff^q(G,E)=0$
for all $q\neq 0$. To get an idea of what vanishing result to expect for representations up
to homotopy, we first point out what happens in the regular case. In this case, one can apply
directly the vanishing for ordinary representations (combined with the spectral sequence of the 
previous subsection).

\begin{corollary}\label{regular proper}
If $E=\bigoplus_{l=a}^b E^l$ is a regular unital representation up
to homotopy of a proper Lie groupoid $G$, then 
\begin{equation*}
\uHdiff^q(G;E)\cong \Gamma({\cal{H}}^q(E))_{inv}.
\end{equation*}
In particular, $\uHdiff^q(G;E)=0$ if $q\notin [a,b]$.
\end{corollary}

\begin{proof}
Since $E$ is regular, we know that each ${\cal{H}}^q(E)$ is an ordinary representation of $G$ and 
the spectral sequence associated to $E$ has $\mathcal{E}_2^{p,q}\cong \uHdiff^p(G;{\cal{H}}^q(E))$. Since
the groupoid is proper, Proposition $1$ of \cite{Cra2} implies that
$\uHdiff(G;{\cal{H}}^q(E))\cong
\uHdiff^0(G;{\cal{H}}^q(E))=\Gamma({\cal{H}}^q(E))_{inv}$. We conclude
that the spectral sequence degenerates at the second page.
\end{proof}

The aim of this section is to prove a similar vanishing result for all representations up to homotopy.

\begin{theorem}\label{vanishing}
Let $E=\bigoplus_{l=a}^b E^l$ be a unital representation up to
homotopy of a proper Lie groupoid $G$. Then:
\begin{equation*}
\uHdiff^q(G;E) =0, \text{         if } q \notin [a,b].
\end{equation*}
\end{theorem}
\begin{proof}
Clearly, we can assume that $a=0$ and we only need to prove
that the cohomology vanishes above degree $b$. It suffices to show
that, in the spectral sequence induced by the filtration (\ref{filtration}),
the term $\mathcal{E}_2^{p,q}$ is zero if either $q>b$ or $p>0$.
In case $q>b$ already the term $\mathcal{E}_0^{p,q}=0$, therefore it only remains to
show the second statement.
Let us first consider the terms $\mathcal{E}_1^{p,q}$, which are the quotient
spaces $Z^{p,q}/B^{p,q}$ where:
\begin{equation*}
Z^{p,q}:=\left\{\eta \in C^p(G;E^q):D_0(\eta)=0 \right\},
\end{equation*}
and
\begin{equation*}
B^{p,q}:=\left\{\eta \in C^p(G;E^q):\eta=D_0(\gamma) \text{ for some } \gamma \in C^p(G;E^{q-1}) \right\}.
\end{equation*}
The operator
\begin{equation*}
d_1^{p,q}:\mathcal{E}_1^{p,q}\rightarrow \mathcal{E}_1^{p+1,q}
\end{equation*}
is given by:
\begin{equation*}
d_1^{p,q}(\overline{\eta})=\overline{D_1(\eta)}.
\end{equation*}
In order to prove that the cohomology of $d_1^{p,q}$ vanishes for positive values of $p$ we will
need to choose a Haar system on $G$. Namely, a family $\mu=\{\mu^x:x \in M\}$ of smooth
measures $\mu^x$ supported on the manifold $t^{-1}(x)$, such that:
\begin{enumerate}
\item For any $f \in C_c^{\infty}(G)$, the formula
\begin{equation*}
I_{\mu}(f)(x)=\int_{t^{-1}(x)}f(g)d\mu^x(g)
\end{equation*}
defines a smooth function $I_{\mu}(f)$ on $M$.
\item The family $\mu$ is left invariant. Explicitly, given an arrow $G\ni g:x \to y$ and a function
$f \in C_c^{\infty}(t^{-1}(y))$, one has
\begin{equation*}
\int_{t^{-1}(x)}f(gh)d\mu^x(h)=\int_{t^{-1}(y)}f(h)d\mu^y(h).
\end{equation*}
\end{enumerate}
It is well-known that Lie groupoids always admit Haar systems, see e.g. \cite{Tu}.
Now, the properness assumption implies the existence of a \textit{cut-off} function $c\in C^{\infty}(M)$ for the Haar system $\mu$. This is a smooth function on $M$ satisfying:
\begin{enumerate}
\item $t:supp(c \circ s)\rightarrow M$ is a proper map.
\item $\int_{t^{-1}(x)}c(s(g))d\mu^x(g)=1.$
\end{enumerate}
The construction of \textit{cut-off} functions on proper groupoids can be found in the appendix of \cite{Tu}.
Next, we use the Haar system and the function $c$ to define the operator:
\begin{equation*}
\kappa: C^{p+1}(G;E^q) \rightarrow
C^{p}(G; E^q)
\end{equation*}
given by:
\begin{equation*}
\kappa(\eta)(g_1,\dots,g_p)=\int_{t^{-1}(x)}\lambda_g \eta(g^{-1},g_1,\dots,g_p)c(s(g))d\mu^x(g),
\end{equation*}
where $x=t(g_1)$ and $\lambda$ is the quasi-action associated to $E$. For $p>0$, take an element $\overline{\eta} \in \mathcal{E}_1^{p,q}$ such that $d_1^{p,q}(\overline{\eta})=0$. We claim
that $\overline{\eta}=d_1^{p-1,q}\overline{\kappa(\eta)}$ and therefore $\mathcal{E}_2^{p,q}=0$.
In the following computation we will omit the super-index in $d_1^{p,q}$ and also the over-lining
when denoting an class in $\mathcal{E}_1^{p,q}$. However, it is important that the computation takes
place in these quotient spaces.
\begin{eqnarray*}
\kappa d_1 (\eta)(g_1,\dots,g_p)&=&\int_{t^{-1}(x)}\lambda_g \left\{ \lambda_{g^{-1}}\eta(g_1,\dots,g_p) - \eta(g^{-1}g_1,\dots,g_p) \right. \\
&&-\sum_{i=1}^{p-1}(-1)^i\eta(g^{-1},g_1,\dots,g_ig_{i+1},\dots,g_p)\\
&&+(-1)^{p+1}\eta(g^{-1},g_1,\dots,g_{p-1})\left. \right\}c(s(g))d\mu^x(g).
\end{eqnarray*}
Using  the associativity equation (\ref{structure equations}) for representation up to homotopy and
the fact that $E$ is a unital representation up to homotopy, we conclude that the expression above equals:
\begin{eqnarray*}
&=&\int_{t^{-1}(x)}\eta(g_1,\dots,g_p)c(s(g))d\mu^x(g)
+\int_{t^{-1}(x)}d_ER_2(g,g^{-1})\eta(g_1,\dots,g_p)c(s(g))d\mu^x(g)\\
&&+\int_{t^{-1}(x)}R_2(g,g^{-1})d_E\eta(g_1,\dots,g_p)c(s(g))d\mu^x(g)
+ \int_{t^{-1}(x)}\lambda_g \left\{ - \eta(g^{-1}g_1,\dots,g_p) \right. \\
&&+\sum_{i=1}^{p-1}(-1)^{i+1}\eta(g^{-1},g_1,\dots,g_ig_{i+1},\dots,g_p)\\
&&+(-1)^{p+1}\eta(g^{-1},g_1,\dots,g_{p-1})\left. \right\}c(s(g))d\mu^x(g).
\end{eqnarray*}
Since the computation takes place in $\mathcal{E}_1^{p,q}$, the second and
third terms are zero. Also, by the normalization condition $(2)$ on the function
$c$, the first term equals $\eta(g_1,\dots,g_p)$. Using the invariance of the Haar
system, the expression becomes:
 \begin{eqnarray*}
&=&\eta(g_1,\dots,g_p)- \int_{t^{-1}(y)}\lambda_{g_1h}\eta(h^{-1},g_2,\dots,g_p) c(s(h))d\mu^y(h) \\
&&+\int_{t^{-1}(x)}  \sum_{i=1}^{p-1}(-1)^{i+1}\eta(g^{-1},g_1,\dots,g_ig_{i+1},\dots,g_p) c(s(g))d\mu^x(g)\\
&&+(-1)^{p+1}\int_{t^{-1}(x)}\eta(g^{-1},g_1,\dots,g_{p-1})c(s(g))d\mu^x(g).
\end{eqnarray*}
Once again, we use the associativity equation (\ref{structure equations}) for representations
up to homotopy and obtain:
\begin{eqnarray*}
&=&\eta(g_1,\dots,g_p)- \int_{t^{-1}(y)}\lambda_{g_1}\circ \lambda_{h}\eta(h^{-1},g_2,\dots,g_p) c(s(h))d\mu^y(h) \\
&&+ \int_{t^{-1}(y)}\left((R_2(g_1,h)d_E+d_ER_2(g_1,h)\right)\eta(h^{-1},g_2,\dots,g_p) c(s(h))d\mu^y(h) \\
&&+\int_{t^{-1}(x)}  \sum_{i=1}^{p-1}(-1)^{i+1}\eta(g^{-1},g_1,\dots,g_ig_{i+1},\dots,g_p) c(s(g))d\mu^x(g)\\
&&+(-1)^{p+1}\int_{t^{-1}(x)}\eta(g^{-1},g_1,\dots,g_{p-1})c(s(g))d\mu^x(g).
\end{eqnarray*}
The second line is zero in $\mathcal{E}_1^{p,q}$, and the sum above becomes:
\begin{equation*}
=\eta(g_1,\dots,g_p)-d_1 \kappa (\eta)(g_1,\dots,g_p).
\end{equation*}
Since the whole expression equals $\kappa d_1(\eta)(g_1,\dots,g_p)=0$, we conclude that $d_1 \kappa (\eta)=\eta$, and the proof is complete.
\end{proof}

\section{The formula of Bott for general Lie groupoids}\label{proof of the formula}

\subsection{The statements}

The aim of this section is to prove a generalization of the formula of Bott and the resulting
spectral sequence from Lie groups to Lie groupoids. To be able to talk about the symmetric
powers of the coadjoint representation, as explained in Example \ref{axiomat},
we will use the axiomatic approach to the symmetric powers (Definition \ref{axiomati}).
So, let $\uS^q$ be a symmetric power operation on representations up to homotopy.

\begin{theorem}\label{main theorem}
Let $G$ be a Lie groupoid. Then, for the cohomology of the horizontal complex
\begin{equation*}
\xymatrix{
\Omega^q(G_0)\ar[r]^-{d_h}&\Omega^q(G_1)\ar[r]^-{d_h}&\cdots
\ar[r]^-{d_h}& \Omega^q(G_p)\ar[r]^-{d_h}&\cdots, }
\end{equation*}
one has
\begin{equation*}
H_{d_h}^p(\Omega^q(G_{\bullet}))\cong
H_{\mathrm{diff}}^{p-q}(G;\uS^q\uAd^*).
\end{equation*}
\end{theorem}

From the vanishing result of Theorem \ref{vanishing}, we deduce:

\begin{corollary}\label{vanishingBS}
Let $G$ be a proper Lie groupoid. Then
\begin{equation*}
H_{d_h}^p(\Omega^q(G_{\bullet}))=0, \text{ if } p>q.
\end{equation*}
\end{corollary}

Also, from the spectral sequence associated to the Bott-Shulman complex
discussed in the preliminaries, we immediately deduce:

\begin{theorem}\label{theorem spectral sequence}
Let $G$ be a Lie groupoid. There is a spectral sequence converging
to the cohomology of $BG$:
\begin{equation*}
E_1^{pq}=H_{\mathrm{diff}}^{p-q}(G; \uS^q(\uAd^*))\Rightarrow
H^{p+q}(BG).
\end{equation*}
Moreover, if the groupoid is proper, the $E_1$  terms of the spectral sequence
vanish bellow the diagonal.
\end{theorem}

\subsection{The proofs}

We start by discussing connections on action groupoids. More precisely, we fix a Lie groupoid $G$ over $M$
and let $A$ be its algebroid. We also fix an Ehresmann connection $\sigma$ on $G$ and we consider the
explicit realization $\textrm{Ad}_{\sigma}$ for the adjoint representation. To simplify notations, we
will omit the subscript $\sigma$.

We consider the category $\textrm{Man}_G$ of all $G$-spaces, i.e. manifolds $P$ together with a
submersion $\nu$ onto $M$ and an action of $G$ on $P \stackrel{\nu}{\rightarrow} M$. For any such $P$, 
we have the induced action groupoid and algebroid (see the preliminaries), denoted
\[ G^P:= G\ltimes P, \ A^P:= A\ltimes P.\]
We first point out the the connection $\sigma$ induces connections $\sigma_P$ on all the groupoids $G_P$. 
Indeed, for the tangent space of $G^P$, one has the following pull-back diagram:
\[ \xymatrix{
T_{(g, p)} G^P \ar[r]^-{d\textrm{pr}_1}\ar[d]_-{ds} &T_gG\ar[d]^-{ds}\\
T_pP \ar[r]^-{d\nu} & T_{\nu(p)}M }\]
Hence, explicitly, writing
\[ T_{(g, p)} G^P= \{ (X_g, V_p)\in T_gG\times T_pP: (ds)_g(X_g)= (d\nu)_p(V_p)\},  \]
we have
\[ \sigma^{P}_{p, g}(V_p)= (\sigma_g(d\nu(V_p)), V_p).\]
Hence, for each $P$, we obtain an explicit realization $\textrm{Ad}_{\sigma^P}$ of the adjoint representation of $G^P$. Again,
we denote it simply by $\textrm{Ad}_{P}$ or even by $\textrm{Ad}$ when no confusion arises.  

First we want to discuss the differentiable cohomology of the groupoids $G_P$ with coefficients $\uS^q \textrm{Ad}^*$ in the
case where $P$ is a free $G$-manifold. If $G$ acts on $P \stackrel{\nu}{\rightarrow} M$, we say that 
$P$ is a free $G$-manifold if the space of orbits $B:= P/G$ is a smooth manifold with the projection
$\pi: P\rmap B$ smooth submersion and 
\[ G\times_M P \rmap P\times_B P , \ \ (g, p)\mapsto (p,gp)\]
is a diffeomorphism, where the left hand side is the fibered product over $s: G\rmap M$ and $\nu:P\rmap M$. One also says
that $\pi: P\rmap B$ is a principal $G$-bundle. 

\begin{proposition}\label{banal case}
Let $P$ be free $G$-manifold in the previous sense, with projection $\pi: P\rmap B$. Then

\begin{equation*}
H_{\mathrm{diff}}^p(G_P;\uS^q\uAd^*)= \left\{
\begin{array}{ll}
\Omega^q(B) &\text{if }\, p=-q,\\
0 & \text{if }\, p> -q.\\
\end{array}
\right. \end{equation*} More precisely, the sequence
$$\xymatrix{
0\ar[r]&\Omega^q(B)\ar[r]^-{\pi^*}& C(G_P;\uS^q\uAd^*)^{-q}\ar[r]^-D
&C(G_P;\uS^q\uAd^*)^{1-q}\ar[r]^-D & \cdots\\}$$
is exact. Here, according to our notations, $\uAd^*= \uAd_{P}^{*}$ is the representation up to homotopy of $G_P$ induced by the connection $\sigma$ and $D$
is the structure operator of $\uS^q\uAd^*$. \\
\end{proposition}

One explanation for the start of this sequence: since the complex $\uS^q\uAd^*$ vanishes in degrees strictly less then $-q$ and
is $\Lambda^qT^*P$ in degree $-q$,
\[ C(G_P;\uS^q\uAd^*)^{-q}= \bigoplus_{k-l= -q} C^k(G_P; (\uS^q\uAd^*)^{-l})= C^0(G_P; \Lambda^qT^*P)= \Omega^q(P),\]
and then $\pi^*$ is just the pull-back of forms. 

\begin{proof}
Since the action is free, we deduce that $\textrm{Ad}_{P}$ is a regular complex
(with cohomology bundle trivial in all degrees, except degree $1$ where it is
the normal bundle to the orbits $TP/\textrm{Im}(\tilde{\rho}_P)\cong \pi^*TB$). We deduce that
the cochain complex $S^q(\textrm{Ad}_{P}^{*})$ is a regular complex, with cohomology 
\begin{equation*}
{\cal{H}}^{k}(S^{q}(\uAd^*(P)))\cong \left\{
\begin{array}{ll}
\pi^*(\Lambda^qT^*B) &\text{if }\, k=-q,\\
0 & \text{if }\, k\neq -q.\\
\end{array}
\right. \end{equation*}
Hence Theorem \ref{push a representation to the cohomology} tells us that:
\begin{itemize}
\item There is an induced structure of representation up to homotopy structure $D_H$ on the
cohomology complex ${\cal{H}}^k(\uS^q(\uAd^*(P)))$. 
\item There is a quasi-isomorphism $\Phi$ between $\uS^q(\uAd^*(P))$ and $({\cal{H}}(S^q(\uAd^*(P))), D_H)$.
\end{itemize}
In our case, since the cohomology complex is concentrated in degree $-q$:
\[ {\cal{H}}(S^q(\uAd^*(P)))\cong \pi^*(\Lambda^qT^*B)[-q],\]
all that $D_H$ contains is the action of the groupoid $G_P$ (the $R_1$-term). On the other
hand, due to the quasi-action axiom on $\uS^q$, we see that this quasi-action is just the tautological one:
the action of an arrow $g: p\rmap gp$ on $\pi^*(\Lambda^qT^*B)$, which should be a map
\[ \Lambda^qT^{*}_{x}B\rmap \Lambda^qT^{*}_{x}B ,\ \ \ x= \pi(p)= \pi(gp),\]
is just the identity. Since $\Phi$ is a quasi-isomorphism, we have
\begin{equation*}
H_{\mathrm{diff}}^p(G_P;\uS^q\uAd^*)\cong
H_{\mathrm{diff}}^p(G_P;{\cal{H}}(\uS^q\uAd^*))\cong
H_{\mathrm{diff}}^p(G_P;\pi^*(\Lambda^q(T^*B))[-q]).
\end{equation*}
Also, since $P$ is proper, we know from Theorem
\ref{vanishing} that the last cohomology vanishes except for degree $p=-q$, where it equals to 
\begin{equation*}
\Gamma(P,\pi^*(\Lambda^q(T^*B)))_{\textrm{inv}}=\Omega^q(B).
\end{equation*}
\end{proof}

Next, we consider a morphisms in $\textrm{Man}_G$, 
\[ f: Q\rmap P.\]
Hence $f$ commutes with the maps into $M$ and it is $G$-equivariant. We show that $f$ induces a strict morphism
\[ f_*: \textrm{Ad}_{Q}\rmap f^* \textrm{Ad}_{P}.\]
Being strict, it means it only has a $0$-component, i.e. $f_*$ will be a map between the underlying cochain complexes. With the identifications for the fibers
\[ \textrm{Ad}_{Q, q}= A_{\nu(q)}\oplus TQ_q,\ f^*\textrm{Ad}_{P, q}= A_{\nu(q)}\oplus T_{f(q)}P,\]
we define $f_*$ to be $(\textrm{Id}, (df))$.

\begin{lemma}\label{the flat maps} $f_*: \textrm{Ad}_{Q}\rmap f^* \textrm{Ad}_{P}$ is a strict morphism of representations up to homotopy of $G^Q$.
\end{lemma}

\begin{proof} We observe that one can describe the structure operators for the adjoint representation
of an action groupoid $G^Q$ in terms of those of $G$. Given the connection $\sigma$ on $G$ there is an 
identification $T(G\ltimes Q)_{(g,q)}\cong A_{t(g)}\oplus TQ_q$. With respect to this identification we have:
\begin{eqnarray*}
\lambda^{G^Q}_{g,q}(\alpha)&=&\lambda^G_g(\alpha),\\
\lambda^{G^Q}_{g,q}(X)&=&(d\mu)_{(g,q)}(\sigma^{Q}_{(g,q)}(X)),\\
R^{G^Q}_2((g,q),(h,l))X&=&R^G_2(g,h)(d\nu)_l(X)).
\end{eqnarray*}
Here $\alpha \in A^Q_{q}\cong A_{\nu(q)}$, $X \in TQ_q$, $\mu$ denotes the action map $G\ltimes P \rightarrow P$
and $\nu: Q\rmap M$ is the map associated to the $G$-space $Q$. 
Using this description one immediately shows that $f_*$ gives a strict map of representations.
\end{proof}

Passing to duals, and using the canonical morphism $C(G^P; E)\to C^*(G^Q; f^*E)$ for $E= (\textrm{Ad}_{P})^*\in \uRep^{\infty}(G^P)$,
we find induced maps of complexes
\begin{equation}\label{f-unu} 
f^*: C(G^P; \textrm{Ad}_{P}^*)\rmap C(G^Q, \textrm{Ad}_{Q}^*).
\end{equation}
It is easy to see that this construction is functorial: if $P\stackrel{f}{\rmap} Q\stackrel{g}{\rmap}R$ are morphisms in $\textrm{Man}_G$,
then $(g\circ f)_*= g_*\circ f_*$.

Using the axioms for $\uS^q$ on strict morphisms, we find that
\begin{equation}\label{f-doi} 
f^*= S^q(f^*): C(G^P; \uS^q(\textrm{Ad}_{P}^*))\rmap C(G^Q, \uS^q(\textrm{Ad}_{Q}^*))
\end{equation}
is a morphism of complexes. 

We will apply this construction to the following standard ``simplicial resolution'' of $M$ by free $G$-manifolds:
\[  
\xymatrix{
\ldots P^{(2)} \ar@ <+9 pt>[r]^{\flat_0} \ar@ <0 pt>[r]^{\flat_1}\ar@ <-9 pt>[r]_{\flat_2} &    P^{(1)}\ar@ <+4 pt>[r]^{\flat_0} \ar@ <-4 pt>[r]_{\flat_1} & P^{(0)}\ar[r]^{\flat_0} & M\ .
}  
\]
Here we use the $G$-spaces already mentioned in Example \ref{action-nerve} of the preliminaries:
\[P^{(m)}= G_{m+1}, \ \ \nu^{(m)}= t: P^{(m)}\rmap M, (g_1, g_2, \ldots , g_{m+1})\mapsto t(g_1)\]
and $G$ acts on $P^{(m)}\stackrel{\nu^{(m)}}{\rmap}M$ by 
\[ g\cdot (g_1, g_2, \ldots , g_{m+1})= (gg_1, g_2, \ldots , g_{m+1}).\]
We denote by $G^{(m)}$ the resulting action groupoid (a groupoid over $P^{(m)}$).

Also, $\flat_i$ are the $G$-equivariant maps
\[ \flat_i= d_{i+1}: P^{(m)}\rmap P^{(m-1)} \ \ \textrm{for}\ i= 0,\dots, m.\]
The discussion above gives us a sequence of complexes
\begin{eqnarray}\label{seq-to-prove}
C(G; \uS^q\textrm{Ad}^*)\stackrel{\flat^{*}}{\rmap} C(G^{(0)}; \uS^q\textrm{Ad}^*) \stackrel{\flat^{*}}{\rmap}C(G^{(1)}; \uS^q\textrm{Ad}^*) \stackrel{\flat^{*}}{\rmap} \ldots 
\end{eqnarray}
where 
\begin{equation*}
\flat^*:C(G^{(m-1)};\uS^q\uAd^*)\rightarrow C(G^{(m)};\uS^q\uAd^*)
\end{equation*}
 are defined by:
\begin{equation*}
\flat^*=\sum_{i=0}^m (-1)^{i}{\flat_i}^*.
\end{equation*}

\begin{proposition} 
The sequence (\ref{seq-to-prove}) is an exact sequence of complexes.
\end{proposition}

\begin{proof}
The fact that $(\flat^*)^2=0$ follows from the simplicial relations. 
For the exactness, we would like to remark that it is a statement 
which does not involve the structure operators $D$. To prove it, we
will produce a contracting homotopy; according to our remark, 
we do not have to care about compatibility with $D$. With that
in mind, note that our general construction of the maps $f_*$ and $f^*$
at the level of complexes,
for $f: Q\rmap P$ a map between two $G$-manifolds, do not really
require $f$ to be equivariant if we do not care about compatibility with the
structure operators $D$. What is important is that $f$ commutes with the projections
$\nu$ into $M$. More precisely, for such an $f$, the $f_*$ of Lemma \ref{the flat maps}
still makes sense as a morphisms of graded vector bundles,
and then the map $f^*$ appearing in (\ref{f-doi}) 
is still well-defined as a map of graded vector spaces. Also, the functoriality of
the construction is still preserved. In our case, we consider
\[ \sigma_0: P^{(m-1)}\rmap P^{(m)}, \ (g_1, \ldots , g_m)\mapsto (1, g_1, \ldots , g_m)\]
and the induced maps
\begin{equation*}
\sigma^*_0= S^q(\sigma^*_0): C(G^{(m)}; S^q(\uAd^*)) \rightarrow
C(G^{(m-1)};S^q(\uAd^*)).
\end{equation*}
Using the last remark on functoriality, we immediately derive the following equations
\begin{equation*} \sigma^*_0 \flat^*_i= \left\{
\begin{array}{ll}
\uId &\text{if } i=0,\\
\flat^*_{i-1} \sigma^*_0   & \text{if } i> 0.\\
\end{array}
\right. \end{equation*}\\
By a standard computation, these imply that $\sigma_{0}^{*}$ induces a homotopy operator:
\begin{eqnarray*}
\sigma^*_0\flat^*+\flat^*\sigma^*_0&=&\sum_{i=0}^{m}(-1)^{i}\sigma^*_0\flat^*_i+
\sum_{i=0}^{m-1}(-1)^{i}\flat^*_i\sigma^*_0\\
&=&\uId+\sum_{i=1}^{m}(-1)^{i}\flat^*_{i-1}\sigma^*_0+
\sum_{i=0}^{m-1}(-1)^{i}\flat^*_i\sigma^*_0
=\uId.\end{eqnarray*}\\
We conclude that sequence is exact.
\end{proof}

The next step is to put together the previous two propositions. In doing so, we mention here that each of the spaces
$P^{(m)}$ is a principal $G$-bundle over $G_m$ with bundle map $\pi= d_0: P^{(m)}\rmap G_m$. Hence we can
organize the various exact sequences coming from the two propositions  into a double complex with co-augmented
rows and columns, one for each $q\geq 0$:
rows and columns:
$$\xymatrix{
 & \vdots &\vdots &\vdots  & \\
0\ar[r]& C(G;\uS^q\uAd^*)^{2-q} \ar[r]^-{{\flat_0}^*}\ar[u]^D &
C(G^{(0)};\uS^q\uAd^*)^{2-q}
\ar[r]^-{\flat^*}\ar[u]^D & C(G^{(1)};\uS^q\uAd^*)^{2-q}\ar[r]^-{\flat^*}\ar[u]^D& \dots\\
0\ar[r]&C(G;\uS^q\uAd^*)^{1-q} \ar[r]^-{{\flat_0}^*}\ar[u]^D &
C(G^{(0)};\uS^q\uAd^*)^{1-q}
\ar[r]^-{\flat^*}\ar[u]^D & C(G^{(1)};\uS^q\uAd^*)^{1-q}\ar[r]^-{\flat^*}\ar[u]^D& \dots\\
0\ar[r]&C(G;\uS^q\uAd^*)^{-q} \ar[r]^-{{\flat_0}^*}\ar[u]^D &
C(G^{(0)};\uS^q\uAd^*)^{-q}
\ar[r]^-{\flat^*}\ar[u]^D & C(G^{(1)};\uS^q\uAd^*)^{-q}\ar[r]^-{\flat^*}\ar[u]^D& \dots\\
& & \Omega^q(M) \ar[u]^{{d_0}^*} \ar[r]^{\delta}
&\Omega^q(G)\ar[u]^{{d_0}^*}
\ar[r]^{\delta}& \dots  \\
& & 0\ar[u]&0\ar[u]&}$$
Here the vertical operator $D$ is 
the differential corresponding to the representations up to homotopy $\uS^q\uAd^*$- as in Proposition \ref{banal case}
which we can now use. The horizontal $\flat^*$ is the map in the previous proposition. 
Since, by Lemma \ref{the flat maps}, each of the maps $\flat^*_i$ commutes with $D$, so does
$\flat^*$. Hence we have a double complex with co-augmentations of the rows and of the columns,
in which all the co-augmented rows and columns are are exact. It follows that the vertical co-augmentation complex
and the horizontal one have isomorphic cohomology (the inclusions by $\flat^{*}_{0}$ and $d_{0}^{*}$, respectively,
into the total complex of the non co-augmented double complex, are quasi-isomorphisms). This proves the theorem.

\end{document}